\documentclass[11pt]{extarticle}
\usepackage{amsmath, amsthm, amssymb, hyperref, color}
\usepackage[shortlabels]{enumitem}
\usepackage{graphicx}
\usepackage[all]{xypic}
\usepackage{makecell}
\usepackage[final]{pdfpages}
\setboolean{@twoside}{false}
\usepackage{pdfpages}
\usepackage{caption}
\usepackage{subcaption}
\usepackage{scalefnt}
\usepackage{verbatim}
\usepackage{xcolor}
\usepackage{array}
\oddsidemargin \evensidemargin
\newtheorem{theorem}{Theorem}
\newtheorem*{theoremm}{Main Theorem}
\numberwithin{theorem}{section}
\newtheorem{proposition}[theorem]{Proposition}
\newtheorem{lemma}[theorem]{Lemma}
\newtheorem{corollary}[theorem]{Corollary}
\newtheorem{definition}[theorem]{Definition}

\newtheorem{remark}[theorem]{Remark}
\newtheorem{example}[theorem]{Example}

\newcommand{\PP}{\mathbb{P}}
\newcommand{\CC}{\mathbb{C}}
\newcommand{\ZZ}{\mathbb{Z}}
\newcommand{\A}{\alpha}
\newcommand{\B}{\beta}
\newcommand{\C}{\gamma}

 \date{}

\title{\textbf{Phylogenetic complexity of the Kimura $3$-parameter model}}

\author{Mateusz Micha{\l}ek and Emanuele Ventura}

\DeclareMathOperator{\rep}{Rep}
\begin{document}

\maketitle

\begin{abstract}
\noindent In algebraic statistics, the Kimura $3$-parameter model is one of the most interesting and classical phylogenetic models. We prove that the ideals associated to this model are generated in degree four, confirming a conjecture by Sturmfels and Sullivant. 

\end{abstract}


\section{Introduction}

\noindent The part of computational biology that models evolution and describes mutations in this process is called {\it phylogenetics} \cite{Ph}. This is a fertile subject witnessing many connections to several parts of mathematics such as algebraic geometry \cite{BW, 4aut}, combinatorics \cite{BHV01, Marysia, JaJCTA}, and representation theory \cite{CFM, Man2}. The methods used in this context of research are powerful and do not only apply to biology, but are employed in several other fields \cite{Para} such as modeling changes of words in languages \cite{DGZ}, literary studies \cite{phylotales} or linguistics itself \cite{GandT} with ideas going back to Darwin \cite{Darwin}.\\
\indent A crucial object in phylogenetics is a \emph{tree model}, which is a parametric family of probability distributions. It consists of a tree $\mathcal T$, a finite set of states $S$ and a family $\mathcal{M}$ of transition matrices, usually given by a linear subspaces of all $|S|\times |S|$ matrices. The case of particular interest is when $S=\{\textnormal{A,C,G,T}\}$, where the basis elements correspond to the four nucleobases of DNA: adenine (A), cytosine (C), guanine (G), and thymine (T). \\
\indent The models for which $\mathcal{M}$ is a proper subspace of matrices reflect some symmetries among elements of $S$. These symmetries are usually encoded by the action of a finite group $G$ on $S$. In these terms, $\mathcal{M}$ can be regarded as the space of $G$-invariant matrices or tensors. Such models constitute a class of interest and they are called {\it equivariant} \cite{DK}. If $G$ is the trivial group, we obtain the general Markov model, corresponding, on the algebraic geometry side, to secant varieties of Segre products. When the elements of $S$ can be identified with those of $G$, the model is called \emph{group-based}. Henceforth we assume $G$ to be abelian. \\
\indent The simplest among the equivariant, and group-based, models is the  {\it Cavender-Farris-Neyman} model. This is the instance for $S=G=\ZZ_2$, the group with two elements. A good understanding of this model from the algebraic geometry point of view has led to tremendous advances in this field. Sturmfels and Sullivant \cite[Theorem 28]{SS} showed that the algebraic varieties arising from it are defined by quadrics. Additionally, Buczy{\'n}ska and Wi{\'s}niewski described many of its remarkable algebro-geometric properties \cite{BW}. Consequently, Sturmfels and Xu \cite{SX}, and Manon \cite{Man3} described the connections of the model to toric degenerations of moduli spaces of rank two vector bundles on marked curves of fixed genus. For more relations to conformal field theory, we refer to \cite{KM14, Man2}. \\
\indent The Cavender-Farris-Neyman model is the simplest among the \emph{hyperbinary} models \cite[Section 3]{BDW}, that are given by $S=G=(\ZZ_2)^n$. The most biologically meaningful example of those is the {\it Kimura $3$-parameter} model; this corresponds to $n=2$. In this case, $S=\{\textnormal{A,C,G,T}\}$, and, moreover, the action of $G$ reflects the pairing between purines (A,G) and pyrimidines (C,T). This model was introduced by Kimura \cite{K81} much before the setting above was developed. Using numerical experiments, Sturmfels and Sullivant conjectured that the ideals of the algebraic varieties associated to this model are generated by polynomials of degree at most four \cite[Conjecture 30]{SS}. The confirmation of this conjecture is the main result of the present article. For any group $G$, Sturmfels and Sullivant defined the {\it phylogenetic complexity} $\phi(G)$ of $G$. 

\begin{definition}[{\bf Phylogenetic complexity} \cite{SS}]

Let $K_{1,n}$ be the star with $n$ leaves, and $X(G, K_{1,n})$ the variety associated to the group-based model.  
Let $\phi(G,K_{1,n})$ be the maximal degree of a generator in a minimal generating set of the ideal $I(X(G,K_{1,n}))$. The phylogenetic complexity $\phi(G)$ of $G$ is $\sup_{n\in \mathbb N} \lbrace \phi(G,K_{1,n})\rbrace$. 

\end{definition}

\noindent In \cite{MV17}, it was shown that for any abelian group $G$, its phylogenetic complexity $\phi(G)$ is finite. The main contribution of this article is a more detailed study of the phylogenetic complexity of $G=\mathbb Z_2 \times \mathbb Z_2$.

\begin{theoremm}\label{maintheorem}

The phylogenetic complexity of the Kimura $3$-parameter model $\phi(\ZZ_2\times \ZZ_2)$ equals four. 

\end{theoremm}

For more interesting results on the Kimura $3$-parameter model we refer to \cite{CFM, CS11, CS08, K12}.

\vspace{3mm}

\noindent {\bf Algebraic varieties associated to a model.}

\noindent We recall the explicit construction of the algebraic variety associated to a model. It is the Zariski closure of the locus of all probability distributions on the states of leaves allowed in the model. \\
\indent A {\it representation} of a model on a tree $\mathcal T$ is an association $\mathcal{E}\rightarrow \mathcal{M}$ of transition matrices to edges $\mathcal{E}$ of $\mathcal T$. The set of all representations is denoted by $\rep(\mathcal T)$. (Here we do not mention the root distribution, since it does not affect the family of probability distributions we obtain.)
To each vertex $v$ of $\mathcal T$ we associate an $|S|$ dimensional vector space $V_v$ with basis $(v_s)_{s\in S}$. We may regard an element of $\mathcal{M}$ associated to an edge $(v_1,v_2)=e\in \mathcal{E}$ as an element of the tensor product $V_{v_1}\otimes V_{v_2}$. We fix a representation $M\in \rep(\mathcal T)$ and an association ${\bf s}: \mathcal{L}\rightarrow S$. Here $\mathcal{L}$ is the set of leaves, i.e.~vertices of degree one, of $\mathcal T$. Following the usual Markov rule, we may compute the probability of $\bf s$: 
$$P(M,{\bf s})=\sum_{\substack{f:\mathcal{V}\rightarrow S\\ f_{|\mathcal L}={\bf s}}}\prod_{(v_1,v_2)\in \mathcal{E}} \Big( M((v_1,v_2))\Big)_{(f(v_1),f(v_2))},$$
where $\mathcal{V}$ is the set of vertices of $\mathcal T$. We may identify ${\bf s}$ with a basis element $\bigotimes_{l\in \mathcal{L}} l_{{\bf s}(l)}$ of $\bigotimes_{l\in\mathcal{L}} V_l$. This provides the map:
$$\Psi:\rep(\mathcal T)\ni M\rightarrow  \sum_{{{\bf s}:\mathcal{L}\rightarrow S}}P(M,{\bf s}) \bigotimes_{l\in \mathcal{L}} l_{{\bf s}(l)}\in\bigotimes_{l\in\mathcal{L}} V_l.$$
\noindent The image of this map is the family of probability distributions described by the model and its Zariski closure is the algebraic variety that represents the model. For group-based models, we denote this variety $X(G, \mathcal T)$, where $G$ is the group defining the model and $\mathcal T$ is the tree as above.
\vspace{3mm}

\noindent {\bf Earlier contributions.}

\noindent Our proof of the main theorem relies on previous results by many authors that we now recall.\\
\indent The first fundamental tool is the {\bf D}iscrete {\bf F}ourier {\bf T}ransform. This is a linear change of coordinates, based on the representation theory of $G$. For special cases in phylogenetics, it was first used by Hendy and Penny \cite{HendyPenny}, and by Erd\"os, Sz\'{e}kely, and Steel \cite{SSE}. In higher generality, it is treated in \cite{JaJalg,SS}. For group-based models, the DFT turns $\Psi$ into a monomial map, proving that the associated algebraic variety $X(G,\mathcal T)$ is a \emph{toric variety}. This translates the classical algebraic problem of finding defining equations of a variety into a combinatorial one. For more information about toric methods we refer to \cite{Cox, Fult, Stks}.\\
\indent Another key result is the reduction from arbitrary trees to the so-called stars or claw-trees $K_{1,n}$, i.e.,~trees with one inner vertex and $n$ leaves. The general procedure for group-based models to obtain ideals arising from arbitrary trees, knowing the ideals for $K_{1,n}$, was discovered in \cite{SS}. Again, this turned out to be very influential, leading, on one hand, to the general constructions of \emph{toric fiber products} \cite{Man2, Sethtfp}, and, on the other, to generalizations for equivariant models \cite{DK}. \\
\indent Combinatorial and computational methods in toric geometry are very well developed. As a starting point in our article we need to compute algebraic invariants of toric varieties embedded in very high dimensional ambient spaces. Here the computer algebra packages \texttt{Normaliz} \cite{Normaliz}, \texttt{4ti2} \cite{4ti2}, along with previous computational results from \cite{DBM} and \cite{SS} are used. In particular, Castenluovo-Mumford regularity plays a crucial role in the proof for $n=6$. These classical invariants are briefly discussed in the Appendix \ref{sec:App}, for the sake of completeness. \\
\indent This work may be also seen in the framework of the {\it stabilisation} of equations of a family of algebraic varieties. Indeed, our proof not only bounds the degrees of the generators, but in principle provides an inductive procedure to obtain all generators in case of $K_{1,n+1}$, assuming the generators for $K_{1,n}$ to be known. Finding equations of an infinite sequence of algebraic varieties, that come naturally in families, is an interesting current theme of research. This usually involves classical varieties such as secants of Segre varieties \cite{DK2} and Grassmannians \cite{DE15}. Indeed, the main result of Draisma and Eggermont in \cite{DE} shows that for equivariant models the associated algebraic variety can always be defined set-theoretically in some bounded degree, once $G$ and $S$ are both fixed. The fact that $\phi(G)$ is finite constitutes the main result of \cite{MV17}. Recently, another ideal-theoretic result was proved by Sam \cite{S16a} showing that the ideal of $k$th secant variety of $d$th Veronese embeddings is generated in bounded degree that is independent of $d$. Interestingly, the ideal-theoretic generation in bounded degree for secants of Segre varieties and Grassmannians are still central open problems. Finiteness issues are strongly connected with the theory of twisted commutative algebras and $\Delta$-modules by Sam and Snowden \cite{SS16}, and the theory of noetherianity by Draisma and Kuttler \cite{DK2}, Hillar and Sullivant \cite{HS12}, and others. \\
\indent Apart from beautiful results of existence, that are quite often non-constructive or very far from optimal, it is of interest finding an explicit description of phylogenetic algebraic varieties. One of the most well-known examples is the {\it salmon conjecture} \cite{A10}, since the prize offered by Allman for the hypothetical solver would be a smoked Copper river salmon. It asks for the description of $\sigma_4(\PP^3\times\PP^3\times\PP^3)$, the algebraic variety representing the general Markov model for $|S|=4$ and $\mathcal T=K_{1,3}$. The generators of the ideal are still unknown, however a set-theoretic description was found by Friedland and Gross \cite{FriGro12}. More recently, Daleo and Hauenstein \cite{DH16} gave a numerical proof of the salmon conjecture. \\
\indent As far as we know, our result is the only ideal-theoretic description, apart from the Jukes-Cantor model, where $|S|=4$ and $\mathcal T$ is an arbitrary tree.\\

\vspace{3mm}

\noindent {\bf Plan of the article.}

\noindent The whole article is devoted solely to the proof of the main theorem. In Section \ref{notation} we introduce the notation that is used throughout the proof. As the proof consists of several parts, some of them very technical, we present the overview of its structure in Section \ref{part0}. The main result is established in Sections \ref{part1} and \ref{part2}.

\section{Preliminaries and notation}\label{notation}

In this section we collect all the notation and terminology we will use in the rest of the paper. We divide this section into paragraphs to facilitate the reading.\\ 
\vspace{2mm}

\noindent {\bf Groups.}\\
\noindent Henceforth we set $G=\mathbb Z_2\times \mathbb Z_2$, unless otherwise stated. We denote the elements of $G$ by $0, \A, \B$, and $\C$. To denote unknown elements of $G$, we use letters $g,x,y,w,p,q \ldots$ We also refer to an unknown element, that is not relevant in a specific argument, with question mark ``?''. \\
\indent Apart from $G$, the most natural groups that enter the picture are the {\it symmetric group} on $n$ leaves $\mathfrak S_n$, the {\it group of flows} $\mathfrak G$, and the {\it automorphism group} $\textnormal{Aut}(G)$. The group of flows is the following. 

\begin{definition}[{\bf Group of flows}]

Let $G$ be a abelian group and $n\in \mathbb N$. The set of flows $\mathfrak G = \left\{ (g_1, \ldots, g_n) | \sum g_i = 0 \right\} $ of length $n$ of $G$ forms a group 
under the componentwise group operation. It is non-canonically isomorphic to the group $G^{n-1}$, the direct product of $n-1$ copies of $G$. 

\end{definition}

\indent The automorphism group of $G$, $\textnormal{Aut}(G)\cong \mathfrak S_3$, is the group of bijective group homomorphisms from $G$ to itself. The automorphism of $G$ specified by $\A \mapsto \A, \B \mapsto \C, \C\mapsto \B$ is simply denoted by $\B \leftrightarrow \C$; similarly for all the other automorphisms of $G$ having a non-trivial fixed element.

\vspace{3mm}

\noindent{\bf The toric variety $X(G , K_{1,n})$.}\\
\noindent For any abelian group $G$, the variety $X(G , K_{1,n})$ is a projective toric variety of dimension $n(|G|-1)$ living in $\mathbb P^{|G|^{n-1}-1}$, where the projective coordinates are in bijection with flows \cite{JaJalg}.\\
\indent Let us recall here its corresponding polytope. Let $M\cong \ZZ^{|G|}$ be the lattice whose basis corresponds to the elements of $G$. Consider $M^n$ with the basis $e_{(i,g)}$ indexed by pairs $(i,g)\in [n]\times G$. We define a map of sets from the group of flows to the lattice, $\psi: \mathfrak G\rightarrow M^n$, by $\psi((g_1,\dots,g_n))=\sum_{i=1}^n e_{(i,g_i)}$. The vertices of the polytope of $X(G , K_{1,n})$ are the images of the flows under the injective map $\psi$. 
\begin{remark}
The family of varieties $X(G , K_{1,n})$ has a wealth of symmetries; the group $\mathfrak S_n$, the {\it group of flows} $\mathfrak G$, and the {\it automorphism group} $\textnormal{Aut}(G)$ all act on the ideals of these varieties. 
\end{remark}

\vspace{3mm}

\noindent {\bf Binomials, tables, and moves. } \\
\noindent Ideals of toric varieties are binomial prime ideals. Thus they admit a minimal generating set of binomials. Binomials may be identified with a pair of tables of the same size, $T_0$ and $T_1$, of elements of $G$, regarded up to row permutation; this is another natural group in this setting which we implicitly take into account. Indeed, a binomial is a pair of monomials and the variables correspond to rows. Given the number of leaves $n$, coordinates are in bijection with flows of length $n$ of $G$. Hence rows are identified with flows of $n$ elements in $G$. 
Columns are in bijection with the $n$ leaves. From the definition of the toric ideals $I(X(G , K_{1,n}))$ \cite{SS}, it follows that a binomial belongs to $I(X(G , K_{1,n}))$ if and only if the two tables representing it are \emph{compatible}, i.e.,~for each $i$, the $i$th column of $T_0$ and the $i$th column of $T_1$ are equal as multisets. We index the columns of a given pair of tables $T_0,T_1$, with $n$ columns, by integers $1\leq i \leq n$. We refer to the element in the $i$th column of row $r$ as $r(i)$. \\
\indent Let $T$ be any table of elements of $G$. The procedure consisting of selecting a subset of rows in $T$ of cardinality at most $d$, 
and replacing it with a compatible set of rows is a {\it move of degree $d$}. A binomial, represented by a pair of tables $T_0, T_1$ of elements of $G$, 
is generated by binomials of degree at most $d$ if and only if there exists a finite sequence of moves of degree $d$ applied to $T_0$ or $T_1$ that transform 
$T_0$ into $T_1$. \\

\begin{example}\label{exmove}

Let $T$ be the table 

$$
T=\begin{bmatrix}
\A & \A & 0  &\dots  & 0\\
0 & \B & \ldots  & \ldots &\ldots \\
\C & 0 & \ldots  & \ldots &\ldots  \\
\ldots & \ldots & \ldots  & \ldots  &\ldots\\
\end{bmatrix}.
$$

\noindent The table $T$ can be transformed by a move of degree three into the table

$$
\tilde T=\begin{bmatrix}
0 & 0 & 0  &\dots  & 0\\
\C & \A & \ldots  & \ldots &\ldots \\
\A & \B & \ldots  & \ldots &\ldots  \\
\ldots & \ldots & \ldots  & \ldots  &\ldots\\
\end{bmatrix}.
$$

\noindent Indeed, the set of the first three rows of $T$ is compatible with the set of the first three rows of $\tilde T$. Note that if the rows in $T$ are flows, then the rows of $\tilde T$ are flows as well. The move described above is denoted by $$\A\A + 0\B + \C 0 = 00 + \C\A + \A\B.$$

\end{example}

\begin{remark}
In the notation for moves, we do not use the indices of the columns involved in the move. Instead, the indices are always clear from the move itself. For instance, the move in Example \ref{exmove} is in columns $1,2$. Also, note that, in general, the columns used for a move do not need to be consecutive. 
\end{remark}

\begin{remark}
The groups $\mathfrak S_n$, the {\it group of flows} $\mathfrak G$, and the {\it automorphism group} $\textnormal{Aut}(G)$  act on the equations of $X(G , K_{1,n})$,  and hence on the tables. The group $\mathfrak S_n$ acts permuting the columns of the pair of tables corresponding to a binomial in the ideal of the variety. The groups $\mathfrak G$ and $\textnormal{Aut}(G)$ act on the entries of the tables in the natural way, i.e., by evaluation. 
\end{remark}
\indent We now introduce one of the most important concepts for our approach. Given a pair of flows, we define a {\it distance} between them, which will enable us to use an inductive procedure on tables. The distance we consider is the classical Hamming distance between two words.

\begin{definition}[{\bf Hamming distance}]

Let $r_0$ and $r_1$ be two flows in $\mathfrak G$: 

$$
r_0 = (g_1+a_1, \ldots, g_n +a_n) \   \mbox{ and } \  r_1 = (g_1, \ldots , g_n). 
$$

\noindent Let $I=\{ \ell \in [n] | a_\ell \neq 0  \}$ and $J= \{ \ell \in [n] | a_\ell = 0  \}$. The multiset $\{a_{\ell}\}_{\ell\in I}$ constitutes the \textnormal{ disagreement string} $a_{\ell_1}\ldots a_{\ell_{|I|}}$ of the pair of flows $r_0$ and $r_1$. The cardinality $|I|$ is the \textnormal{Hamming distance} between $r_0$ and $r_1$. The multiset $\{a_{\ell}\}_{\ell\in J}$ constitutes their \textnormal{agreement string}. Up to the action of the group of flows $\mathfrak G$ on both flows, we may assume that the group elements $g_i = 0$ for all $i$. 

\end{definition}

\begin{remark}[{\bf Tables and Hamming distance}]

Given a pair of tables $T_0,T_1$, we ``compare'' them using the notion of Hamming distance as follows. Since the tables come with undistinguishable rows, 
we may choose as first rows of $T_0$ and $T_1$ two rows that minimize the Hamming distance among all the pairs of rows from $T_0$ and $T_1$. 
After fixing the first row in $T_0$ and in $T_1$, as described in Section \ref{part0}, 
one of the techniques adopted in Sections \ref{part1} and \ref{part2} is as follows. With moves of degree at most four, we create another pair of rows with strictly smaller Hamming distance than the initial one. 

\end{remark}

\vspace{3mm}

\noindent {\bf Counting functions.} \\
\noindent We will make use of counting functions on the tables $T_0$ and $T_1$. A counting function $f$ on the columns of $T_0$ has the same values as counting function on the columns of $T_1$, since the pairs of tables we are interested in are compatible, i.e., columnwise they are the same as multisets. Given $x\in G$, we denote by $x_{i_1 \ldots i_k}$ the number of copies $x\in G$ appearing in the columns $i_1,\ldots,i_k$ in $T_0$, or in $T_1$. 

\begin{example}
The function $\A_{12}-2\cdot 0_3$ counts the number of copies of $\A$ in columns $1$ and $2$ minus two times the number of copies of $0$ in column $3$.
\end{example}

From an algebraic point of view, a counting function defines a {\it grading} of the variables, that is a specialization of the multi-grading. Thus the fact that the counting function gives the same value on two tables is equivalent to the fact that the two corresponding monomials have the same degree with respect to the induced grading. Additionally, from the perspective of toric geometry, the counting function is induced by restricting the torus action to a special one-parameter subgroup. \\
\vspace{3mm}

\noindent {\bf Group homomorphisms}. \\
\noindent We will make use of group homomorphisms in order to do counting arguments in a given pair of tables. We denote 
$$
\phi_g: \mathbb Z_2 \times \mathbb Z_2 \rightarrow \mathbb Z_2,
$$
\noindent the group homomorphism given by the quotient map sending each element $x\in G$ to its class modulo the subgroup generated by the element $g\in G$.

\section{Complexity of the Kimura $3$-parameter model}

The aim of this section is to establish the phylogenetic complexity of the Kimura $3$-parameter model. In Section \ref{part0}, we discuss the structure of the proof, postponing the technical part of it to Sections \ref{part1} and \ref{part2}.

\subsection{Main result and structure of the proof}\label{part0}

\noindent We proceed presenting our main result along with the outline of the plan of the proof strategy. 

\begin{theorem}

The phylogenetic complexity of the Kimura $3$-parameter model $\phi(\ZZ_2\times \ZZ_2)$ equals four.

\end{theorem}

\begin{figure}
\caption{Matryoshka of the proof.}
\label{fig1}
\begin{center}
\includegraphics[scale=0.6]{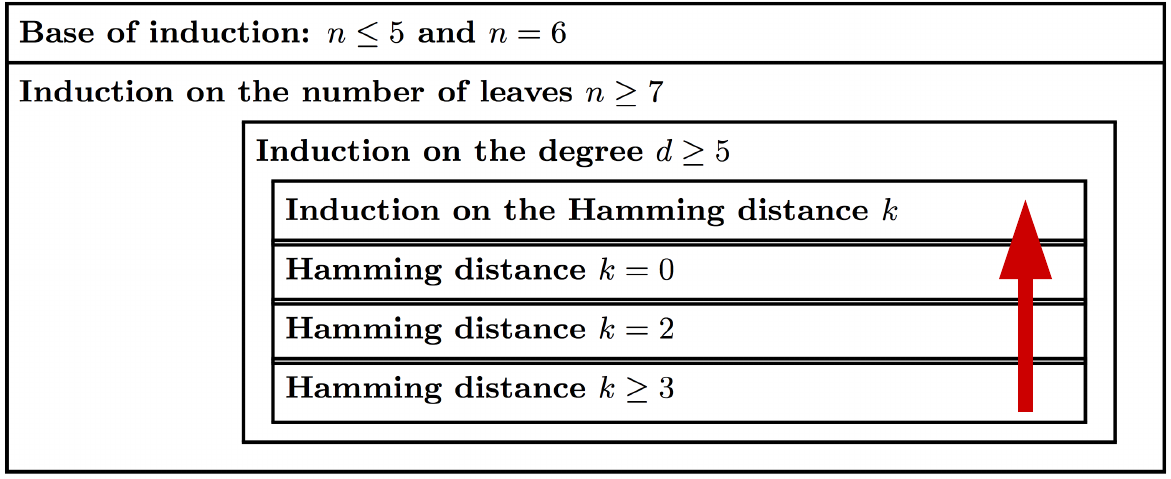}
\end{center}
\end{figure}

\indent The structure of the proof is presented in Figure \ref{fig1}. Our proof is an induction on the number of leaves $n$, i.e.,~the number of columns of the tables. The base of our induction is $n=3$. 
The case of $n\leq 5$ leaves has been studied computationally. More precisely, for $n=3$ the result is presented in \cite{SS} and for $n=4$ it is computed in \cite{DBM}. For $n=5$ we used the program featured in \cite{DBM} to produce the vertices of the polytope.  The computer algebra program \texttt{4ti2} \cite{4ti2} specialized for toric ideals was able to compute the Markov basis using a server equipped with a CPU \texttt{4 Intel-Xeon E7-8837/32 cores/2.67GHz} and a memory of \texttt{1024Gb} RAM. 

\begin{proposition}

The ideal $I(X(G, K_{1,5}))$ is minimally generated by $22240$ polynomials of degree at most four: $12960$ quadrics, $2560$ cubics, and $6720$ quartics. 

\end{proposition}
\indent The case $n=6$ is treated in Section \ref{part3}. Methods similar to the general case $n\geq 7$ and bounds on Castelnuovo-Mumford regularity obtained using \texttt{Normaliz} \cite{Normaliz} allow us to reduce the problem to a computation handled with \texttt{4ti2}. From the computational point of view, it is interesting to note that we were not able to address the case $n=6$ only with computational tools. Based on our experiments with \texttt{4ti2}, we expect the computation to be not feasible: it would run for several years on a server of the same capability as the one mentioned above, and a memory of \texttt{1Tb} RAM would not be sufficient to finish the computation. \\
\indent For $n\geq 7$, we have an induction on the degree $d$ of the generators, i.e.,~the number of rows of the table. Inside a specific degree $d$, we have an induction on the Hamming distance $k$ of two rows of the tables. The strategy in this inner induction on the Hamming distance $k$ is the following. Suppose we have a binomial generator of degree $d\geq 5$. Hence, we have a pair of tables consisting of $d$ rows each and with $n\geq 7$ columns. Two rows have Hamming distance $k$ and we reduce it to $k=0$; in other words, the given pair of tables is transformed into a pair of tables that have an identical row. This is a binomial which is a product of a binomial of degree $d-1$ and a variable. By induction on $d$, such a binomial can be generated in degree at most $4$.\\
\indent Hence the aim of the induction on the Hamming distance $k$ is to reduce it to $k=0$. In order to achieve this, we address the case $k\geq 3$ into two separate propositions in Section \ref{part1}; see Proposition \ref{aaaaa,aaaa,aaabc,aabb,aaab} and Corollary \ref{Reduction of four}, and Proposition \ref{abc}. This reduces the proof to $k=2$. Recall that there do not exist flows whose Hamming distance is $k=1$, since they cannot disagree only in one entry. \\
\indent We now discuss the strategy in case $k=2$, the technical heart of the proof, which is tackled in Section \ref{part2}. In spite of many symmetries, discussed in Section \ref{notation}, there are several cases one has to consider: We identify ten cases, indexed by roman numerals, where the first two rows of the given pair of tables $T_0, T_1$ have a disagreement string of length $k=2$. Here we provide a uniform proof for three crucial cases: {\bf Case I}, {\bf II}, and {\bf III}. As we show them simultaneously with the very same techniques, we refer to those as the ``main case''. The rest of the cases is treated by reducing them to the main case. \\
\indent For the proof in the main case, we look at the second rows of each of the tables $T_0$ and $T_1$. Let $\ell$ denote the length of
the disagreement string between those two, in columns {\it not} involving the first two. By Corollary  \ref{Reduction of four}, we are able to assume $\ell \leq 3$ and, since $n\geq 7$, the length of the agreement string between the second row of $T_0$ and the second row of $T_1$, outside columns $1$ and $2$, is at least $n-5\geq 2$. Since the columns are indistinguishable up to the action of $\mathfrak S_n$, we may assume that the columns $n-1$ and $n$ are involved in the agreement string.  Now the aim is to reduce to the situation in which no row has two nonzero entries in the columns $n-1$ and $n$: employing moves of degree at most four, we would like to eliminate
all the strings which have nonzero entries on both columns $n-1$ and $n$. We call such strings {\it bad pairs}. 

\begin{definition}[{\bf Bad pairs}]

A \textnormal{bad pair} is a string $xy$, where the elements $x,y\in G$ are such that: 

\begin{enumerate}
\item[(i)] they are both nonzero; 
\item[(ii)] $x$ is in column $n-1$ and $y$ is in column $n$. 
\end{enumerate}

\end{definition}

We now show that eliminating all the bad pairs we fall back to the case of $n-1$ leaves, which allows us to conclude, by the outermost induction.

\begin{theorem}\label{n->n-1 by badpairs}

Suppose that a pair of compatible tables $T_0, T_1$ with $n\geq 7$ columns do not contain rows with bad pairs. Then the corresponding binomial is generated in degree at most $\phi(G, K_{1,n-1})$. 

\begin{proof}

The assumption implies that for every row $r$ of $T_0$ and $T_1$ we have either $r(n-1)=0$ or $r(n)=0$. Summing up the columns $n-1$ and $n$, we obtain two tables $\tilde T_0$ and $\tilde T_1$.
The crucial observation is that $\tilde T_0$ and $\tilde T_1$ are compatible tables with $n-1$ columns. Hence they correspond to a binomial 
in $I(X(G , K_{1,n-1}))$. This binomial is generated in degree at most $\phi(G, K_{1,n-1})$ by definition. This implies that $\tilde T_0$ 
and $\tilde T_1$ can be transformed into each other by a finite sequence of moves of degree at most $\phi(G, K_{1,n-1})$. Each of these moves 
lifts to the tables $T_0$ and $T_1$, transforming all their columns accordingly, except columns $n-1$ and $n$. Here the moves permute 
the pairs of elements, where each pair is formed by the two elements in columns $n-1$ and $n$, in a fixed row. These moves transform $T_0, T_1$ into 
$\hat{T}_0, \hat{T}_1$. The latter need not be the same though; indeed, they may differ in columns $n-1$ and $n$. As in the proof of \cite[Theorem 3.12]{MV17}, we make quadratic moves to adjust the elements in columns $n-1$ and $n$. These transform $\hat{T}_0$ into $\hat{T}_1$. Hence the tables $T_0, T_1$ are generated in degree at most $\phi(G, K_{1,n-1})$. 
\end{proof}

\end{theorem}

\begin{figure}
\caption{Zoom in of Hamming distance $k=2$ step.}
\label{fig2}
\begin{center}
\includegraphics[scale=0.7]{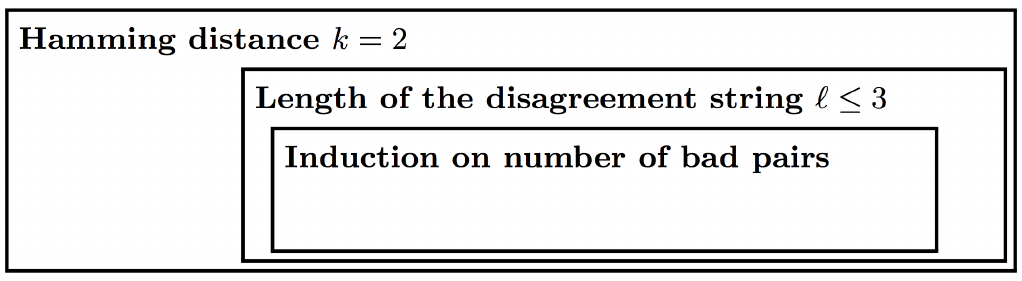}
\end{center}
\end{figure}

\subsection{Reduction of Hamming distance $\geq$ 3}\label{part1}

In this section, we start our reduction of the Hamming distance. More precisely, we assume the Hamming distance to be at least  three and we prove that we can reduce it to two; the latter will be discussed in Section \ref{part2}. We proceed analyzing the cases when the disagreement string is given by at least four entries. 

\begin{proposition}\label{aaaaa,aaaa,aaabc,aabb,aaab}

The disagreement strings  (i) $\A\A\A\A$, (ii) $\A\A\B\B$, (iii) $\A\A\B\C$, and (iv) $\A\A\A\B$ can be reduced. 

\begin{proof}

\noindent (i). Consider the function $0_{1234} - \A_{1234}$. By the action of the group of flows $\mathfrak G$, we may assume that this counting function is nonpositive on both of the tables. Since the function is stricly positive in the first row of $T_1$, there exists a row $r$ in $T_1$ where there are strictly more copies of $\A$ than copies of $0$ in the columns $1,2,3,4$. On the other hand, $r$ cannot contain $\A\A$ in two of the columns $1,2,3,4$, since we would exchange those with the corresponding entries in the first row and this would decrease the Hamming distance. Thus $r$ has one copy of $\A$ and no copies of $0$ in columns $1,2,3,4$. If the row $r$ has both copies of $\B$ and $\C$, we would move the string $\A\B\C$ to the first row of $T_1$, reducing the Hamming distance. Whence we may assume that $r$ contains the string $\A\B\B\B$ in columns $1,2,3,4$. Notice that in columns $2, 3, 4$ of $T_1$, there are no strings of the form $\A\A$ or $\C\C$, otherwise quadratic moves would decrease the Hamming distance. Additionally, in columns $2,3,4$ there is no string of the form  $\A\C$; for this we can apply in $T_1$ the cubic move $0000+\A\B\B\B+?\A\C = \A\B\C 0 +0\A 0\B + ?0\B$. Now, we introduce the counting function $0_{234}+\B_{234}-\A_{234}-\C_{234}$ on $T_1$. By the previous discussion about the possible strings in columns $2,3,4$, this function is at least one in every row of $T_1$. Consequently, there exists a row $r'$  in $T_0$ where this function is three. As a consequence, the row $r'$ contains either the string $\B\B$ or $00$. This would decrease the Hamming distance.\\

\noindent (ii). Consider the counting function $0_{1234}-\A_{12}-\B_{34}$. By the action of the group of flows $\mathfrak G$, we may assume it is nonpositive on both of the tables. Since this function is strictly positive on the first row of $T_1$, there exists a row $r$ in $T_1$ where the function is strictly negative. Note that  on the row $r$, one has $\A_{12}, \B_{34}\leq 1 $; otherwise we would make a quadratic move, involving $r$ and the first row of $T_1$, reducing the Hamming distance. \\
\indent If in the row $r$ we have $\A_{12}=\B_{34}=1$, then $0_{1234} \leq 1$, by the value of the counting function on $r$. Hence in the row $r$, there exists $\C$, which allows us to make a quadratic move reducing the Hamming distance. Without loss of generality, we have $\A_{12}=1, \B_{34}=0$, and $0_{1234}=0$. Thus the row $r$ contains either the string $\C\A\C\C$ or the string $\A\C\C\C$. In both cases, we exchange $\C\C$ with the first row of $T_1$ and we act with the flow $(0, 0 , \C, \C)$ on $T_0$ producing $\A\A\A\A$, which is (i). \\

\noindent (iii). Consider the function $0_{1234}-\A_{12}-\B_3-\C_4$. By the action of the group of flows $\mathfrak G$, we may assume it is nonpositive on both of the tables. Therefore there exists a row $r$ in $T_1$ where the function is strictly positive. Note that on the row $r$ one has $\A_{12}\leq 1$. \\
\indent If in the row $r$ we have $\A_{12}=1$ and $\B_3=1$, then we may assume $r$ contains the string $\A x\B y$ in columns $1,2,3,4$. We have $x,y\neq\A,\B,\C$, as otherwise in each of these circumstances we would make a quadratic move between $r$ and the first row of $T_1$, reducing the Hamming distance. Then the function is zero on $r$, which is not possible by assumption. Analogously, we may conclude when $\A_{12}=1$ and $\C_4=1$. \\
\indent If in the row $r$ we have $\A_{12}=1, \B_3=0$, and $\C_4=0$, then $0_{1234}=0$. In this case we have $\A_{34}=0$, because of a quadratic move between $r$ and the first row of $T_1$. Hence the row $r$ contains the string $\A \C\B$ in columns $1,3,4$, which again would reduce the Hamming distance.\\
\indent If in the row $r$ we have $\A_{12}=0$, then either $\B_3=1$ or $\C_4=1$. If $\B_3=\C_4=1$, then in columns $1,2$ the row $r$ contains the string $00$; indeed we cannot have copies of $\A, \B$ or $\C$ by quadratic moves with the first row of $T_1$. This implies that the counting function $\A_{12}+\B_3+\C_4-0_{1234}$ is zero on the row $r$, which is not possible by the assumption. If in the row $r$ we have $\B_3=0$ and $\C_4=1$, then $0_{1234}=0$. In the row $r$ we can now exclude all the possible elements in each column by quadratic moves, obtaining the string $\B\A\C$ in columns $2,3,4$. We exchange this string with the first row of $T_1$, reducing the Hamming distance. Analogously, if in the row $r$ we have $\B_3=1$ and $\C_4=0$, we obtain $\C\B\A$ in columns $2,3,4$, and we conclude in the same way. \\

\noindent (iv). Consider the counting function $0_{1234}-\A_{123}-\B_4$. By the action of the group of flows $\mathfrak G$, we may assume it is nonpositive on the tables. Therefore there exists a row $r$ in $T_1$ where the function is strictly negative. Thus on the row $r$ we have $0_{1234} \leq 1$, as $\A_{123}\leq 1$. \\
\indent Suppose that in the row $r$ we have $0_{1234}=1$. Then $\A_{123}=1$ and $\B_4=1$, by the assumption on the value of the counting function on $r$. In two of the columns $1,2,3$ we cannot have $\A$ or $\B$ by quadratic moves, involving $r$ and the first row of $T_1$. Thus we have a copy of $\C$; we now make a quadratic move between $r$ and the first row of $T_1$, which decrease the Hamming distance.\\
\indent Suppose that in the row $r$ we have $0_{1234}=0$. If in the row $r$ we have $\A_{123}=0$, then $\B_4=1$. In columns $1,2,3$ we cannot have $\B$, as otherwise we would exchange the string $\B\B$ with the first row of $T_1$, thus reducing the Hamming distance. Whence $r$ contains the string $\C\C\C\B$ in columns $1,2,3,4$. If in the row $r$ we have $\A_{123}=1$, then $\B_4=0$. In this situation, by the same argument, $r$ contains the string $\A\C\C\C$ (or $\C\A\C\C$ or $\C\C\A\C$). We claim that having the string $\A\C\C\C$ can be reduced to the case of having the string $\C\C\C\B$ up to quadratic moves and group automorphism. 
Indeed, suppose we have the string $\A\C\C\C$ in the row $r$. We exchange $\C\C$ from $r$ with $00$ from the first row of $T_1$ in columns $2,3$. We act with the flow $(0,\C,\C,0)$ on both tables and we transpose column $1$ and column $4$. Now the row $r$ contains the string $\C\C\C\B$ in columns $1,2,3,4$. \\
\indent By the previous discussion, it is enough to deal only with the string $\C\C\C\B$ in $r$. Consider the counting function $\A_{123}+\B_{123}-\C_{123}-0_{123}$. Note that this function has only odd values. We now show that  the function cannot be positive on a row of $T_1$. Indeed, assume there is a row $r'$ where the function takes a positive value. Then the row $r'$ contains either $\A\A$, $\B\B$ or $\A\B$ in columns $1, 2, 3$. The first two cases are not possible, because we would exchange them with the string $\C\C$ in the row $r$; this would produce $\A\A$ or $\B\B$ in the row $r$, which we would exchange with $00$ in the first row of $T_1$. 
We are left with the possibility of $r'$ having $\A\B$ in columns $1,2,3$. For this we apply in $T_1$ the cubic move $0000+\C\C\C\B + ?\A\B = \C\C\B\B+0\A 00 + ?0\C$. \\
\indent In conclusion, the counting function $\A_{123}+\B_{123}-\C_{123}-0_{123}$ is strictly negative on every row of $T_1$. Since the value of this function on the first row of $T_0$ is $3$, there exists a row $r''$ in $T_0$ on which the function is $-3$. Thus in $r''$ we have either $00$ or $\C\C$ in columns $1,2,3$. In this case, we would exchange them with the first row of $T_0$ reducing the Hamming distance. 
\end{proof}

\end{proposition}

\begin{corollary}\label{Reduction of four}

Suppose that a table $T$ contains two rows $r$ and $r'$ having disagreement string of cardinality four. Then, using moves if degree at most three, $T$ can be transformed in such a way that the disagreement string has cardinality at most three. Moreover, only the four columns of the disagreement string are involved in the reduction. 

\begin{proof}

Assume two rows $r$ and $r'$ do not agree on four elements. Up to the action of the group of flows $\mathfrak G$ and $\mathfrak S_4$, the elements of $r$ in the disagreement string can be set to be $0000$; all the possibilities for the elements of $r'$ in the disagreement string are $\A\A\A\A$, $\A\A\B\B$, $\A\A\B\C$, and $\A\A\A\B$. By Proposition \ref{aaaaa,aaaa,aaabc,aabb,aaab}, these disagreement strings can be reduced. Hence, performing the moves in the proof of the Proposition \ref{aaaaa,aaaa,aaabc,aabb,aaab}, we transform the tables in such a way that the cardinality of the disagreement string is at most three. 
\end{proof}

\end{corollary}

Now we deal with the disagreement string of length three, $\A\B\C$. We begin with preparatory lemmas. 

\begin{lemma}\label{exists 00 in abc}

Suppose that the disagreement string between $T_0$ and $T_1$ is $\A\B\C$, in columns $1,2,3$. Then we may assume that there exists a row $r'$ in $T_0$ containing the string $00$ in columns $1,2,3$.  

\begin{proof}

We introduce the counting function $0_{123}-\A_1-\B_2-\C_3$. By the action of the group of flows $\mathfrak G$, we may assume that the sum is nonnegative on $T_0$. Then there exists a row $r'$ in $T_0$ where the function is strictly positive. \\
\indent If in the row $r'$ we have $0_{123}=1$, then $\A_1=\B_2=\C_3=0$, by the assumption on the counting function evaluated at $r'$. By the action of the group of flows $\mathfrak G$, we may assume without loss of generality that $r'$ contains the string $0xy$ in columns $1,2,3$. Then $x\neq 0, \B$ by assumption.  Also, $x\neq \C$, as otherwise we would exchange the string $0\C$ with $\A\B$ in the first row of $T_0$, reducing the Hamming distance between $T_0$ and $T_1$. Hence $x=\A$. Similarly, $y\neq 0,\C$ and $y\neq \B$, as otherwise we exchange $0\B$ with $\A\C$ in the first row of $T_0$. Hence $r'$ contains the string $0\A\A$ in columns $1,2,3$, which we exchange with the first row of $T_0$. 
\end{proof}

\end{lemma}

\begin{lemma}\label{00c in abc}

We may assume that the row $r'$ of Lemma \ref{exists 00 in abc} in $T_0$ contains the string $00\C$ in columns $1,2,3$. More generally, for every row $r''$ containing the string $00$ in columns $1,2,3$, the nonzero element of $r''$ in columns $1,2,3$ coincides with the corresponding entry of the first row of $T_0$. 
\begin{proof}
The row $r'$ contains a string with exactly two elements equal to $0$ in the columns $1,2$ and $3$. By the action of $\mathfrak S_n$, we may assume that $r'$ contains the string 
$00x$ in columns $1,2,3$. Note that $x\neq \A,\B$, as in both cases we make a quadratic move between $r'$ and the first row of $T_0$, reducing the Hamming 
distance between $T_0$ and $T_1$. Thus $x=\C$. By the action of the group of flows $\mathfrak G$, in every row $r''$ containing the string $00$ in columns $1,2,3$, the nonzero entry coincides with the corresponding entry of the first row of $T_0$. 
\end{proof}
\end{lemma}

\begin{lemma}\label{00c implies in abc}

Suppose that in $T_0$ we have a row $r'$ containing $00\C$. Then this is the only string that a row with $00$ in columns $1,2,3$ may contain. 

\begin{proof}
Since the row $r'$ in $T_0$ contains $00\C$, then it cannot contain another copy of $\C$, as we would exchange with the first row of $T_0$, 
thus reducing the Hamming distance. Hence $r'$ contains $00\C\A\B$, since it is a flow.  Assume there exists another row $r''$ containing a string with $00$, different from $00\C$. By Lemma  \ref{00c in abc}, the unique nonzero entry in columns $1,2,3$ of $r''$ agrees with the corresponding entry of the first row of $T_0$. Assume that $r''$ contains $0\B0$ in columns $1,2,3$. Then we apply the cubic move $\A\B\C 00+00\C\A\B+0\B 0 = 0\B 00\B+0\B\C \A0+\A 0\C$, reducing the Hamming distance. For a row containing $\A 00$ we conclude in the same way. 
\end{proof}

\end{lemma}

\begin{lemma}\label{ab0ab in abc}

As in the proof of Lemma \ref{00c implies in abc}, we assume that $r'$ contains $00\C\A\B$ in columns $1,2,3,4,5$. 
There exists a row $r''$ in $T_0$ such that $r''(3) = 0$ and, moreover, $r''$ contains the string $\A\B 0 \A\B$ in columns $1,2,3,4,5$. 

\begin{proof}

Such a row $r''$ exists in $T_0$ by the compatibility of the two tables. The structure of $T_0$ is: 
$$
T_0=\begin{bmatrix}
\A & \B & \C & 0  & 0 &\ldots & 0\\
0 & 0 & \C  & \A &\B &\ldots & \ldots \\
x & y & 0  & z &w &\ldots & \ldots \\
\ldots & \ldots & \ldots  & \ldots &\ldots  &\ldots & \ldots\\
\end{bmatrix}.
$$
\noindent By Lemma \ref{00c implies in abc}, we have $x,y\neq 0$. Analogously, we have $z,w\neq 0$ by applying Lemma \ref{00c implies in abc}, upon
exchanging the string $\A\B\C 00$ in the first row with $00\C\A\B$ in the second row.\\
\noindent Note that  $x\neq \B$ and $y\neq \A$, as otherwise, exchanging with the first row, in the first case with $\A\C$ and in the second with $\B\C$, we would reduce the Hamming distance; analogously, $z\neq \B$ and $w\neq \A$. Furthermore, by Lemma \ref{00c in abc}, we have $x,y \neq \C$ as otherwise we would create
the string $\C 00$ and $0\C 0$ respectively. Analogously $z,w\neq \C$. Hence the only remaining possibility is $xy=\A\B$ and $zw=\A\B$. \end{proof}
\end{lemma}

\begin{lemma}\label{counting funct at most -1 in abc}

The counting function $0_{12345}-\A_{14}-\B_{25}-\C_3$ is at most $-1$ on every row of $T_0$. 

\begin{proof}

For the sake of contradiction, suppose there exists a row $r$ in $T_0$, where the counting function is nonnegative. In $T_0$, there exists a row $r''$ with $r''(3) = 0$. By Lemma \ref{ab0ab in abc}, the row $r''$ contains the string $\A\B 0\A\B$. \\
\indent If in the row $r$ we have $0_{12345} \geq 3$, then $r(3)\neq 0$, again, by Lemma \ref{ab0ab in abc}. Hence we have at least two differences
with $r''$ and we can make a quadratic move between $r$ and $r''$. This reduces the Hamming distance. Thus on the row $r$ one has $0_{12345}\leq 2$. \\
\indent If $0_{12345}=2$ on $r$, we have the following possibilities:

\begin{enumerate}

\item[(i)] $r$ contains $00\C xy$;

\item[(ii)] $r$ contains $xyz00$;

\item[(iii)] $r$ contains $x0yz0$;

\item[(iv)] $r$ contains $x0y0z$.

\end{enumerate}

\indent In case (i), we have $x, y\neq 0$ by the assumption on the value of the counting function.  Additionally, $x,y\neq \C$, as we would exchange the string $00\C\C$ with the first row in $T_0$. Consider the differences between $r$ and $r''$. If $xy\neq \A\B$, then we can make a move involving column $3$, at most one of columns $1,2$ and either column $4$ or $5$ between $r$ and $r''$. This allows us to exchange $\C$ in $r$ with $0$ in $r''$; this contradicts Lemma \ref{ab0ab in abc}. Hence $xy = \A\B$, which on the other hand contradicts the nonnegativity of the counting function. Exchanging $r'$, the row appearing in Lemma \ref{ab0ab in abc} 
containing $00\C\A\B$, with the first row of $T_0$, case (ii) is the same as case (i). \\
\indent In case (iii), $x\neq 0$ by the assumption on the value of the counting function. Moreover, $x\neq \C$ since we would exchange $\C 0$ in $r$ with the string $\A\B$ in $r''$ in columns $1,2$, contradicting Lemma \ref{00c in abc}. We also have $x\neq \B$, because we could make a quadratic move in columns $1,5$ between $\B 0$ in $r$ with $0\B$ in $r'$, obtaining the string $\B 0 \C\A 0$ in $r'$. Now, we exchange in columns $1,3$, the string $\B\C$ in $r'$ with $\A 0$ in $r''$, which produces the string $\A 00\A 0$; this reduces the Hamming distance. Finally, if $x=\A$, we exchange in columns $1,2,5$, the string $\A 00$ in $r$ with $\A\B \B$ in $r''$, obtaining $\A 00\A 0$, which again reduces the Hamming distance. \\
\indent In case (iv), $x\neq 0$ by the assumption on the value of the counting function. Additionally, $x\neq \C$, because otherwise we would exchange in columns $1,2$ the string $\C 0$ in $r$ with $\A\B$ in $r''$, thus contradicting Lemma \ref{ab0ab in abc}. Also, $x\neq \B$, as we would make a quadratic move on columns $1,2,4$ between $r$ and $r''$, contradicting again Lemma \ref{ab0ab in abc}. Analogously, $z\neq 0, \A,\C$. Hence $r$ contains the string $\A 0y0\B$. We exchange in columns $1,2,4,5$ the string $\A 0 0\B$ in $r$ with $00\A\B$ in $r'$, which produces $00y$ in $r$, which in turn implies $y=\C$ by Lemma \ref{00c in abc}. This contradicts the nonnegativity of the counting function. \\
\indent If $0_{12345}=1$, by symmetry, we may assume $r(2)=0$ or $r(3)=0$. If $r(3)=0$, then by Lemma \ref{ab0ab in abc}, $r$ contains $\A\B 0\A\B$, which contradicts the nonnegativity of the counting function. If $r(2)=0$, then $r$ contains $x0yzt$. Then $z\neq 0$ by the assumption. Moreover, $z\neq \C$, as we would exchange $r$ with $r''$ in columns $2,4$, contradicting Lemma \ref{ab0ab in abc}. \\
\indent If $z=\A$, we now consider the value of $x$. We have $x\neq 0$ by assumption. We have $x\neq \A$ by assumption on the nonnegativity of the counting function. Moreover, $x\neq \C$, since otherwise we would exchange in columns $1,2$ the string $\C 0$ in $r$ with $\A\B$ in $r''$ contradicting Lemma \ref{ab0ab in abc}. Hence $x=\B$, i.e., $r$ contains the string $\B 0y\A t$. Now, $t\neq 0$, by the assumption on the value of $0_{12345}$. Moreover $t\neq \B$, by the assumption on the value of the counting function on $r$. Also notice that $t\neq \C$, as otherwise we exchange in columns $1,5$ the string $\B\C$ of $r$ with $\A 0$ of the first row of $T_0$, and then we exchange $\B\B$ from the first row with $00$ in $r'$ reducing the Hamming distance. Therefore $r$ contains the string $\B 0y\A\A$, which we exchange with the string $\A\B 0\A\B$ in $r''$ in columns $1$ and $5$, contradicting Lemma \ref{ab0ab in abc}. 

If $z=\B$, then $r$ contains $x0y\B t$. Furthermore, $t\neq 0$ by assumption on the value of $0_{12345}$. Moreover, $t\neq \A$ exchanging in columns $4,5$ the string $\B\A$ of $r$ with $\A\B$ of $r''$, contradicting Lemma \ref{ab0ab in abc}. Analogously, we would contradict Lemma \ref{ab0ab in abc} for $t=\C$, exchanging in columns $2,4,5$, the string 
$0\B\C$ in $r$ with $\B\A\B$ in $r''$. Hence $r$ contains the string $x0y\B\B$. Here $x\neq 0$, by assumption. Moreover, $x\neq \A$, because of the nonnegativity of the counting function. Also, $x\neq \C$, because we would contradict Lemma \ref{ab0ab in abc}, exchanging $\C 0$ of $r$ with $\A\B$ of $r''$. Therefore $r$ contains $\B 0 y\B\B$, but we exchange it with $\A\B 0\A\B$ in columns $1,4$ contradicting Lemma \ref{ab0ab in abc}. \\
\indent If $0_{12345}=0$, then $\A_{14}=\B_{25}=\C_3=0$, by the assumption on the nonnegativity of the function on $r$. Thus $r$ contains $xyztw$ different from $\A\B 0\A\B$ in columns $1,2,4,5$. Hence we have two identical differences between $r$ and $r''$, which allow to make a quadratic move, contradicting Lemma \ref{ab0ab in abc}. 
\end{proof}

\end{lemma}

\begin{proposition}\label{abc}

The disagreement string $\A\B\C$ can be reduced. 

\begin{proof}

By Lemma \ref{counting funct at most -1 in abc}, the counting function $0_{12345}-\A_{14}-\B_{25}-\C_3$ is at most $-1$ on every row of $T_0$. As a consequence, there exists a row $r$ in $T_1$, where the function is at most $-2$. By the value of the counting function on the row $r$, the entries in $r$ must agree in two, three, four or five entries with $\A\B\C\A\B$. \\
\indent If $r$ agrees in five entries, it contains $\A\B\C\A\B$. We exchange $\A\B\C$ with $000$ in the first row of $T_1$, which reduces the Hamming distance between $T_0$ and $T_1$. If $r$ agrees in four entries, we denote by $x$ the element where $r$ does not agree with $\A\B\C\A\B$. If $x\neq r(3)$, then 
we would have either the string $\A\B\C$ or $\C\A\B$, which is also in table $T_0$; this reduces the Hamming distance. Suppose $r$ contains $\A \B x \A\B$. If $x=0$, the table $T_0$ contains the same flow. If $x=\A$ or $\B$, we exchange $\A\A$ or $\B\B$ with $00$ in the first row of $T_1$. \\
\indent If $r$ agrees with $\A\B\C\A\B$ in three entries, we denote by $xy$ the remaining two. First, note that if $xy$ are in columns $1,2$ or in columns $4,5$, we exchange $\A\B\C$ or $\C\A\B$ with $000$ in the first row of $T_1$; this decreases the Hamming distance. \\
\indent  Assume that both of $x$ and $y$ are in columns $1,2,3$. If $r$ contains $x\B y\A\B$, then $x\neq \A, \B$, because otherwise we would exchange the string $\A\A$ or $\B\B$ with the first row of $T_1$ reducing the Hamming distance. Whence $x=0,\C$. Moreover $y\neq \C$, by definition. Additionally, $y\neq \B$, because we would move $\B\B$ to the first row of $T_1$, reducing the Hamming distance. It follows that $y=0,\A$. On the other hand, $xy\neq 00$, since the counting function $0_{12345}-\A_{14}-\B_{25}-\C_3$ is at most $-2$ on $r$. Furthermore, $x+y\neq \B$, as otherwise we would exchange $x\B y$ with $000$ in the first row of $T_1$, reducing the Hamming distance between $T_0$ and $T_1$. Hence $r$ contains either $\C\B 0 \A\B$ or $0\B \A \A \B$. For the first, we exchange in columns $2,3,5$, the string $\B 0\B$ with $000$ in the first row of $T_1$, and we exchange $\A\B 0 \A\B$ in $T_0$ with the first row of $T_0$. For the second, we exchange $0\B \A \A \B$ with the first row of $T_1$ and $\A\B 0 \A\B$ in $T_0$ with the first row of $T_0$, which reduces the Hamming distance. \\
\indent If $r$ contains $\A x y\A\B$, then applying the automorphism $\A \leftrightarrow \B$ and a transposition between columns $1$ and $2$, we are in the case when the row $r$ contains $x\B y\A\B$. \\
\indent If  $x,y$ are both in columns $3,4,5$, we apply analogous moves as the ones featured above. Then we may assume that $x$ is either in column $1$ or $2$, and $y$ is either in column $4$ or $5$. In all these cases, we have $x=0$ and $y=0$, as all the other possibilities are excluded by exchanging with the first row of $T_1$. The fact that $x=y=0$ contradicts the value of the counting function on $r$. \\
\indent If $r$ agrees with $\A\B\C\A\B$ in two entries, we have $0_{12345}=0$ on $r$, since the value of the counting function $0_{12345}-\A_{14}-\B_{25}-\C_3$ on $r$ is at most $-2$. In columns $1,2,3$, there is at least one entry $x$ which does not agree with the corresponding entry in $\A\B\C$, because otherwise we would move $\A\B\C$ to the first row of $T_1$, reducing the Hamming distance. Denoting the elements where they do not agree by $x,y,z$, the strings that $r$ may contain are: $\A xy\A z$, $\A\B xyz$, and $\A xyz\B$. Note that these are all the possible, as the remaining ones are resolved in the same way upon exchanging the string $\A\B\C 00$ in the first row with $00\C\A\B$ in the second row of $T_0$. 
If $r$ contains $\A xy\A z$, then we exchange the string $\A \A$ of $r$ in columns $1,4$ with $00$ in $T_1$. We now exchange the string $\A\B 0\A\B$ of $r''$ with the first row in $T_0$; these two rows have lower Hamming distance. If $r$ contains $\A\B x$ in columns $1,2,3$, then $x\neq \C$, by the counting function. Moreover, $x\neq 0$ since $0_{12345}=0$. Hence $x=\A$ or $\B$. Now we exchange $\A\A$ or $\B\B$ with $00$ in the first row of $T_1$ reducing the Hamming distance. If $r$ contains $\A xyz \B$, by definition or by quadratic moves we can exclude the cases $x=\A, \B, 0$, and $y=\A, \C, 0$. Hence $r$ contains $\A \C \B$, which we exchange with the first row of $T_1$, decreasing the Hamming distance. 
\end{proof}

\end{proposition}

The preceding results of this section show the following corollary. 

\begin{corollary}\label{hamming distance at most 2}

The Hamming distance of two flows can be reduced to at most two. 

\end{corollary}

\subsection{The disagreement string $\A\A$}\label{part2}

In this section, we proceed in the case of the disagreement string $\A\A$. 

\begin{equation}\label{initialtables}
T_0-T_1=\begin{bmatrix}
\A & \A & 0  &\dots  & 0\\
0 & x & \ldots  & \ldots &\ldots \\
y & 0 & \ldots  & \ldots &\ldots  \\
\ldots & \ldots & \ldots  & \ldots  &\ldots\\
\end{bmatrix}
-
\begin{bmatrix}
0 & 0 & 0  & \ldots & 0 \\
 \A & z & \ldots  & \ldots & \ldots \\
w & \A & \ldots  & \ldots & \ldots \\
\ldots & \ldots  & \ldots & \ldots & \ldots\\
\end{bmatrix}.
\end{equation}

\noindent Let us denote the row in $T_0$ starting with the string $0x$ by $r_{0x}$ and the row in $T_1$ starting with the string $\A z$ by $r_{\A z}$.  After fixing the first rows and the first two columns, we make moves of degree at most four on the rest of tables in such a way that the number of agreements in $r_{0x}$ and $r_{\A z}$ is maximized. 

\begin{remark}\label{there exists two entries that agree}

Corollary \ref{Reduction of four} ensures that, after possibly making moves of degree at most four, the rows $r_{0x}$ and $r_{\A z}$ in $T_0$ and $T_1$ respectively, agree in at least $n-5$ entries. Up to the action of $\mathfrak S_n$ on the $n$ leaves, and hence on the columns, these are the last $n-5$ columns. 

\begin{definition}\label{agreementstring}

The string in the last $n-5$ columns of the rows $r_{0x}$ and $r_{\A z}$ is the \emph{the agreement string} between $r_{0x}$ and $r_{\A z}$. Up to the action of the group of flows $\mathfrak G$, these entries are zeros. 

\end{definition}

\end{remark}

Our aim is to prove the following three crucial cases, which we refer to as the {\it main case}: \\
\begin{equation}\label{maincase}
\begin{matrix}  \tag{$\star$}
\textnormal{\noindent {\bf Case I}}: x=\B, y=\A, z=\B, w=0; \\
\textnormal{\noindent {\bf Case II}}: x=\B, y=\A, z=\B, w=\B; \\
\textnormal{\noindent {\bf Case III}}: x=\B, y=\B, z=\B, w=\B.\\
\end{matrix}
\end{equation}

\noindent In Section \ref{subsec:redto3}, we reduce any other possible case to one of the above. 

\subsubsection{Reduction to the main case}\label{subsec:redto3}

\indent Up to the action of the group of flows $\mathfrak G$, there are at least as many copies of $0$ as copies of $\A$ in the first two columns of $T_0$. Up to the action of $\textnormal{Aut}(G)$, we may assume $x=\B$. We will show that all cases can be resolved, by reducing to the main case (\ref{maincase}).\\
\indent We first collect a useful lemma which we will use to resolve easily some of the cases. 

\begin{lemma}\label{if w,z=b,c}

If in table $T_1$ in \textnormal{(\ref{initialtables})} we have $\left\{z,w\right\} = \left\{\B, \C \right\}$, then the corresponding cases can be reduced. If in table $T_0$ in \textnormal{(\ref{initialtables})} we have $\left\{x,y\right\} = \left\{\B, \C \right\}$, then the corresponding cases can be reduced.

\begin{proof}
If $\left\{z,w\right\} = \left\{\B, \C \right\}$, then in $T_1$ we have either the cubic move $00+\A \B + \C \A = \A\A + 0\B + \C 0$ or $00+\A \C + \B \A = \A\A + \B 0 + \C 0$. The second sentence is the symmetric version of the first: acting with the flow $(\A,\A,0,\ldots, 0)\in \mathfrak G$ on the tables, we produce the same tables as in the first statement. 
\end{proof}

\end{lemma}

We now analyze all the possible cases. We refer to the tables $T_0$ and $T_1$ in (\ref{initialtables}). 

\noindent {\bf Case $y=\A$.} In this case, the table $T_0$ has the form: 
$$
T_0=\begin{bmatrix}
\A & \A & 0  &\dots  & 0\\
0 & \B & \ldots  & \ldots &\ldots \\
\A & 0 & \ldots  & \ldots &\ldots  \\
\ldots & \ldots & \ldots  & \ldots  &\ldots\\
\end{bmatrix}.
$$

\noindent We may have $z=0,\B,\C$.

\noindent {\bf $z=\B$}. \\
Here, $w=\C$ is reduced by Lemma \ref{if w,z=b,c}. Hence we have $w=0$ ({\bf Case I}) or $w=\B$ ({\bf Case II}). \\

\noindent {\bf $z=0$}. \\
Here, $w=0$ ({\bf Case X}), $w=\B$ ({\bf Case VII}), $w=\C$ ({\bf Case VI}). \\

\noindent {\bf $z=\C$}. \\
Here, $w=0$ ({\bf Case IV}), $w=\C$ ({\bf Case V}), $w=\B$ is resolved by Lemma \ref{if w,z=b,c}.\\

\noindent {\bf Case $y=\B$}. In this case, the table $T_0$ has the form: 
$$
T_0=\begin{bmatrix}
\A & \A & 0 &\dots  & 0\\
0 & \B & \ldots  & \ldots &\ldots \\
\B & 0 & \ldots & \ldots &\ldots  \\
\ldots & \ldots & \ldots & \ldots  &\ldots\\
\end{bmatrix}.
$$

\noindent We may have $z=\B,0,\C$. \\

\noindent {\bf $z=\B$}. \\
Here, $w=0$ (which is {\bf Case II} by acting with the flow $(\A,\A,0,\ldots,0)\in \mathfrak G$ and $\C\leftrightarrow \B$), $w=\B$ ({\bf Case III}), $w=\C$ resolved by Lemma \ref{if w,z=b,c}.\\

\noindent {\bf $z=0$}.\\
Here, $w=0$ ({\bf Case IX}), $w=\B$ (which is {\bf Case II} by acting with the flow $(\A,\A,0,\ldots,0)\in \mathfrak G$, transposing and $\C\leftrightarrow \B$), $w=\C$ (which is {\bf Case V} by acting the flow $(\A,\A,0,\ldots,0)\in \mathfrak G$ and transposition). \\

\noindent {\bf $z=\C$}.\\
Here, $w=0$ (which is {\bf Case V} by acting with the flow $(\A,\A,0,\ldots,0)\in \mathfrak G$), $w=\C$ ({\bf Case VIII}).\\

We now reduce all the cases to the main case (\ref{maincase}), postponing its proof for the moment, as this requires more technical results. \\

\noindent {\bf Cases IV and V.} \\
\noindent In this case we have: 
$$
T_0-T_1=\begin{bmatrix}
\A & \A & 0 &\dots  & 0\\
0 & \B & \ldots  & \ldots &\ldots \\
\A & 0 & \ldots  & \ldots &\ldots  \\
\ldots & \ldots & \ldots & \ldots  &\ldots\\
\end{bmatrix}
-
\begin{bmatrix}
0 & 0 & 0 & \ldots & 0 \\
\A & \C & \ldots & \ldots & \ldots  \\
\ldots & \ldots & \ldots & \ldots & \ldots \\
\ldots & \ldots & \ldots & \ldots & \ldots \\
\end{bmatrix}.
$$

\noindent We may assume we do not have strings $\C\C, 00, \C 0, 0\C$ in columns $1,2$ of $T_0$; this is shown by the same arguments in the proof of Lemma \ref{lem:forbidden configurations}. Hence the counting function $\A_{12}+\B_{12}-0_{12}-\C_{12}$ is nonnegative on every row of $T_0$. On the other hand, in the table $T_1$, in columns $1,2$ we do not have the string $\A\B$, as we would reduce this case with a cubic move. In the same columns of $T_1$, the string $\A\A$ would decrease the Hamming distance. Moreover, the string $\B\B$ is reduced by the cubic move $\A\C+0\A+\B\B=0\B+\B\C+\A\A$, and $\B\A$ is reduced by the cubic move $00+\A\C+\B\A=\A\A+\B 0+0\C$. This is a contradiction and thus it shows the reduction. \\

\noindent {\bf Case VI.}\\
In this case we have: 
$$
T_0-T_1=\begin{bmatrix}
\A & \A & 0 & \ldots & 0\\
0 & \B & \ldots & \ldots & \ldots \\
\A & 0 & \ldots & \ldots & \ldots  \\
\ldots & \ldots & \ldots & \ldots & \ldots \\
\end{bmatrix}
-
\begin{bmatrix}
0 & 0 & 0 & \ldots & 0 \\
 \A & 0 & \ldots & \ldots & \ldots  \\
\C & \A & \ldots & \ldots & \ldots  \\
\ldots & \ldots & \ldots & \ldots & \ldots \\
\end{bmatrix}.
$$

\noindent In columns $1,2$ in $T_1$, the string $\A\B$ is resolved by Lemma \ref{if w,z=b,c}.  The string $\A\C$ in columns $1,2$ of $T_1$ is {\bf Case V}. Since we cannot have the string $\A\A$ in columns $1,2$ of $T_1$, the counting function $\A_1-0_2$ is nonpositive in every row of $T_1$. Thus there exists a row $r$ in $T_0$ with $r(2)=0$ and $r(1)\neq \A$. Hence $r(1)=\B$. Acting by the flow $(\A,\A,0,\ldots,0)$ and transposition we reduce to {\bf Case V}. \\

\noindent {\bf Case VII.}\\
In this case we have: 

$$
T_0-T_1=\begin{bmatrix}
\A & \A & 0 & \ldots & 0\\
0 & \B & \ldots & \ldots & \ldots \\
\A & 0 & \ldots & \ldots & \ldots \\
\ldots & \ldots & \ldots & \ldots & \ldots  \\
\end{bmatrix}
-
\begin{bmatrix}
0 & 0 & 0 & \ldots & 0 \\
\A & 0 & \ldots & \ldots & \ldots \\
\B & \A & \ldots & \ldots & \ldots  \\
\ldots & \ldots & \ldots & \ldots & \ldots \\
\end{bmatrix}.
$$

\noindent We exclude the string $\A\B$ in columns $1,2$ in $T_1$, since it is {\bf Case II}. We also exclude $\A\C$ by Lemma \ref{if w,z=b,c}. As in {\bf Case VI}, there exists a row $r$ in $T_0$ such that $r(1)=\B$ and $r(2)=0$. Now, by acting with the flow $(\A,\A,0,\ldots,0)\in \mathfrak G$, making a transposition and applying the group automorphism $\C\leftrightarrow \B$, we reduce to {\bf Case II}. \\

\noindent {\bf Case VIII.} \\
In this case we have:

$$
T_0-T_1=\begin{bmatrix}
\A & \A & 0 &\dots  & 0\\
0 & \B & \ldots & \ldots & \ldots \\
\B & 0 & \ldots & \ldots & \ldots  \\
\ldots & \ldots & \ldots & \ldots & \ldots  \\
\end{bmatrix}
-
\begin{bmatrix}
0 & 0 & 0 & \ldots & 0 \\
\A & \C & \ldots & \ldots & \ldots \\
\C & \A & \ldots & \ldots & \ldots \\
\ldots & \ldots & \ldots & \ldots & \ldots \\
\end{bmatrix}.
$$
\noindent We may exclude in columns $1,2$ in $T_0$ the string $00$. Also, we exclude the string $\C\C$ by the quartic move $\A\A+0\B+\B 0 +\C\C = 00+\A\C + \C\A + \B\B$. Moreover, in columns $1,2$ in $T_0$, notice that we can exclude the strings $0\C$ and $\C 0$ by Lemma \ref{if w,z=b,c}. Hence the counting function $\A_{12}+\B_{12}-0_{12}-\C_{12}$ is nonnegative on every row of $T_0$. On the other hand, in $T_1$ we may reduce the string $\A\A$, $\A\B$ and $\B\A$ by Lemma \ref{if w,z=b,c}. Finally, we are able to reduce the string $\B\B$ by the quartic move $00+\A\C+\C\A+\B\B = \C\C+\B 0+0\B+\A\A$. This is a contradiction and thus it shows the reduction. \\

\noindent {\bf Case IX.} \\
In this case we have:
$$
T_0-T_1=\begin{bmatrix}
\A & \A & 0 &\dots  & 0\\
0 & \B & \ldots & \ldots & \ldots  \\
\B & 0 & \ldots & \ldots & \ldots  \\
\ldots & \ldots & \ldots & \ldots & \ldots  \\
\end{bmatrix}
-
\begin{bmatrix}
0 & 0 & 0 & \ldots & 0 \\
 \A & 0 & \ldots & \ldots & \ldots  \\
0 & \A & \ldots & \ldots & \ldots  \\
\ldots & \ldots & \ldots & \ldots & \ldots \\
\end{bmatrix}.
$$

\noindent Analogously to the proof of Lemma \ref{lem:forbidden configurations}, we exclude $\C\C,0\C,\C 0,00$ in columns $1,2$ of $T_0$. So the counting function $\A_{12}+\B_{12}-0_{12}-\C_{12}$ is nonnegative on every row of $T_0$. On the other hand, in columns $1,2$ of $T_1$, the strings $\A\B,\B\A$ correspond to the case for $z=\B$ and $w=0$ in tables (\ref{initialtables}), which were previously done. Thus there exists a row $r$ such that $r(1)=\B$ and $r(2)=\B$ by the positivity of the counting function in $T_0$ and $T_1$. Exchanging the string $00$ in the first row with the string $\B\B$ in $r$, acting by $\A\A$ on both $T_0$ and $T_1$, applying the automorphisms $\C \leftrightarrow \A$ and $\C \leftrightarrow \B$ we obtain {\bf Case III}.\\

\noindent {\bf Case X.} \\
In this case we have:
$$
T_0-T_1=\begin{bmatrix}
\A & \A & 0 &\dots  & 0\\
0 & \B & \ldots & \ldots & \ldots \\
\A & 0 & \ldots & \ldots & \ldots  \\
\ldots & \ldots & \ldots & \ldots & \ldots  \\
\end{bmatrix}
-
\begin{bmatrix}
0 & 0 & 0 & \ldots & 0 \\
 \A & 0 & \ldots & \ldots & \ldots\\
0 & \A & \ldots & \ldots & \ldots \\
\ldots & \ldots & \ldots & \ldots & \ldots \\
\end{bmatrix}.
$$

\noindent In $T_1$, in columns $1,2$ we can exclude $\A\B$, because it is {\bf Case I}. The string $\A\C$ reduces to {\bf Case IV}. As usual, the string $\A\A$ is excluded. Hence the counting function $\A_1 -0_2$ is nonpositive in every row of $T_1$. Hence there exists a row $r$ in $T_0$ such that $r(1)\neq \A$ and $r(2)=0$. The possible values of $r(1)$ are either $\C$ or $\B$, since for $r(1)=0$ we have an immediate reduction. For $r(1)=\C$ we apply Lemma \ref{if w,z=b,c} and $r(1)=\B$ is {\bf Case IX}. 

\subsubsection{Preliminary Lemmas}

We are now ready to present our preliminary lemmas, that are devised to tackle the main case (\ref{maincase}). As they will be used very often, we give them specific reference names in order to facilitate the reading. 

\begin{lemma}[{\bf Difference Lemma}]\label{lem:difference}

Suppose we have the table $T$ whose first three rows are $r_1,r_2,r_3$: 
\[
T = \begin{bmatrix}
q&q&\ldots\\
x&y&\ldots\\
z&x&\ldots\\
\ldots & \ldots & \ldots 
\end{bmatrix},\]

\noindent where $q,x,y,z\in G$ and $x\neq y,z$. If one of the following holds: 

\begin{enumerate}
\item[(i)] $z\neq y$ and $r_2(i)-r_3(i)$ is $x-y$ or $x-z$ for some $i>2$; or
\item[(ii)] $z=y$, $q\neq y$ and $r_2(i)-r_3(i)$ is $x-y$ or $x-q$ for some $i>2$,
\end{enumerate}
\noindent then we can transform the row $r_1$ to a row starting with the string $xx$. 
\begin{proof}
 When the difference $r_2(i)-r_3(i) = x-y$ in both (i) and (ii), we make the quadratic move $xyw+zx(w+x-y)=xx(w+x-y)+zyw$, which exchanges the corresponding 
entries in rows $r_2$ and $r_3$, thus creating a row starting with the string $xx$. Analogously for the case (i), when the difference is $r_2(i)-r_3(i)=x-z$. In (ii), when the difference is $r_2(i)-r_3(i) = x-q $, we make the cubic move $qq + xyw + yx(w+x-q) = xx + qy(w+x-q)+yqw$.
\end{proof}
\end{lemma}


\begin{remark}

Note that the Difference Lemma \ref{lem:difference} distinguishes one group element in each table in each of the crucial cases {\bf Case I}, {\bf Case II}, and {\bf Case III}. In all the cases, these are $\C$ in $T_0$ and $\B$ in $T_1$. In particular, if the second and third row differ on some index $i>2$, then their difference must be equal to the distinguished element. 

\end{remark}

Although basic, the Difference Lemma \ref{lem:difference} will be used very frequently. We apply it following the observation above. Indeed, our aim will be often to produce a row starting with a string of type $xx$ and conclude by induction. To this end, after identifying the situation described in Lemma \ref{lem:difference}, if $r_2(i)\neq r_3(i)$, then we will be able to immediately infer {\it what can be} the element $r_2(i)-r_3(i)\in G$; to exclude all the other possible values we apply the Difference Lemma \ref{lem:difference}, obtaining a row starting with the string $xx$. This will be useful to decrease the given Hamming distance and conclude by induction on the degree.

\begin{lemma}[{\bf Standard Lemma}]\label{standard argument}

Let $T$ be a table and suppose there is an element $y\in G$ in some row $r$ with $r(n-1)=0$ and $r(n)=0$. Suppose there is a row $r'$ of $T$ with $r'(n-1)=x$ and $r'(n)=x$, where $0\neq x\in G$, and a row $r''$ with the element $y+x$ in the same column as $y$. Then we can exchange $y$ and $y+x$ (and appropriate entries in columns $n-1$ and $n$). The same statement holds when $y$ is a string of elements of $G$. 
\end{lemma}
\begin{proof}

Let us consider the entries $r''(n-1)=u$ and $r''(n)=v$. If $u=x$ or $v=x$, then we make a quadratic move putting $y+x$ and $x$ in the row $r$. 
If $u=0$ or $v=0$, then we move the string $xx$ to the row $r$, and finally we exchange $x$ with $0$ and $y$ with $y+x$. If $u=v$ are equal, then we move the string $xx$ in the row $r''$ exchanging it with the string $uv$, thus we exchange $y$ with $y+x$ and $0$ with $x$. Hence, we may assume that $u\neq v$ and they are both different from $0$ and $x$. Hence, the sum $u+v+x=0$. Thus, we may exchange $y$ with $y+x$ and $00$ with $uv$. The last statement is shown using the same arguments. This completes the proof.
 \end{proof}

We now record some technical results on the main case (\ref{maincase}). Note that in the main case we have $x=z=\B$. 

\begin{lemma}\label{lem:agreeon0}
If $r_{0\B}(i)=r_{\A \B}(i)$ for some $i>2$ then we may assume that both are equal to $0\in G$. In particular, both rows have $0$ on the agreement string.
\end{lemma}
\begin{proof}
Without loss of generality, let us assume $i=3$. If $r_{0\B}(i)=r_{\A \B}(i)=\B$, then a quadratic move allows us to produce the string $0\B\B 0$ in both tables. If $r_{0\B}(i)=r_{\A \B}(i)=\C$, then in both tables we obtain the string $\A\B\C 0$ by quadratic moves again. If $r_{0\B}(i)=r_{\A \B}(i)=\A$, then in both tables we obtain $\A 0\A 0$. The last string is obtained in $T_1$ by quadratic moves, and in $T_0$ by the following moves:
\begin{enumerate}
\item[(i)] in {\bf Case I} and {\bf II}, by the cubic move $\A\A0+0\B\A+\A 0=\A 0 \A+\A\B 0+0\A$;
\item[(ii)] in {\bf Case III}, by two quadratic moves, upon exchanging $0\B$ with $\B 0$. 
\end{enumerate}
\end{proof}

\begin{remark}\label{rem:symmetry}

We observe that in {\bf Case I} and {\bf Case III}, the tables $T_0$ and $T_1$ are in ``symmetry''. More precisely, the fixed entries in table $T_1$ can be obtained from the ones in $T_0$, by acting with the flow $(\A,\A, 0,\dots,0)\in \mathfrak G$ and applying the automorphism $\B\leftrightarrow\C$ of $G$, that exchanges $\B$ and $\C$. In particular, if we can prove a statement for $T_0$ then a ``symmetric'' statement holds for $T_1$.

\end{remark}

\begin{lemma}\label{lem:forbidden configurations}
We may assume that no row in $T_0$ contains in columns $1,2$ any string of the form $\C\C,0\C,\C 0,00$. Analogously, no row in $T_1$ contains in columns $1,2$ any of the strings of the form $\C\C,\A\C,\C\A,\A\A$.
\begin{proof}
In all the cases, one can obtain either $00$ in $T_0$ or $\A\A$ in $T_1$. For $T_0$, these are: $\A\A+0\B+\A 0+0\C=00+\A\C+0\A  +\A \B$, $\A\A+0\B+\C0=00+\C\A+\A\B$, $\A\A+0\B+\B 0+\C\C=00+\C\A+\A\C+\B\B$. The statement for $T_1$ readily follows by Remark \ref{rem:symmetry}.
\end{proof}
\end{lemma}

\begin{lemma}\label{lem:differ not by c}
For any row $r_{xy}$ in $T_0$ differing from $r_{0\B}$ on some column index $i>2$ \textnormal{not} by $\C$, we may assume $xy=\A\A,\B\B,\A\B$ or $\B\A$. Analogously, in $T_1$, if $r_{xy}$ differs from $r_{\A\B}$ on some column index $i>2$ \textnormal{not} by $\B$, then $xy=00,\B\B,\B 0$ or $0\B$.

\begin{proof}
In {\bf Case I} and {\bf Case II}, by the Difference Lemma \ref{lem:difference}, and a quadratic move with $r_{0\B}$ or with $r_{\A 0}$ in $T_0$, we may assume $x+y=0$ or $x+y=\C$. The result follows by Lemma \ref{lem:forbidden configurations}.\\
\indent In {\bf Case III} we exclude $x+y=\A$. Indeed, if $xy=\A 0$ or $xy=0\A$, then we are in {\bf Case II} (more precisely, for $0\A$ we also need to exchange the two columns to reduce to {\bf Case II}). If $xy=\B\C$ or $\C\B$, by the quadratic moves $0\B w+\B\C (w+\A)=0\C (w+\A)+\B\B w$ or $0\B w+\C\B (w+\B)=\B\B (w+\B)+\C 0 w$ we produce $0\C$ or $\C 0$ and apply Lemma \ref{if w,z=b,c}. Remark \ref{rem:symmetry} gives the symmetric statement for $T_1$. 
 \end{proof}
\end{lemma}

\begin{lemma}\label{lem:b->0}

If there exists and index $j$ such that $r_{0\B}(j)=\B$ in $T_0$, then we may assume that $r_{\A\B}(j)=0$ in  $T_1$. Analogously, if there exists an index $j$ such that $r_{\A \B}(j)=\C$ in $T_1$, then we may assume that $r_{0\B}(j)=0$. 

\end{lemma}

\begin{proof}

Assume $r_{0\B}(3)=\B$. Suppose $r_{\A\B}(3)= \A$ or $\C$ in $T_1$. Then there exists a row $r$ in $T_1$ with $r(3)=\B$. The row $r$ contains the string $xy\B$ in columns $1,2,3$ for some $x,y\in G$. Let us determine the possible values of $r(2)=y$. If $y=\B$ we would have the string $0\B\B$ in both of the tables. 
By Lemma \ref{lem:differ not by c}, $y=0$. Whence the counting function $\B_3 -0_2$ is nonpositive on every row of $T_1$. It follows that in $T_0$ there exists a row $r'$ with $r'(2)=0$ and $r'(3)\neq \B$. By Lemma \ref{lem:differ not by c}, we have $r'(3)=\A$. For $r_{\A\B}(3)= \A$, by quadratic moves, we obtain $\A 0 \A 0$ in both tables. Now consider the case $r_{\A\B}(3)=\C$. In $T_1$ for every row $r$ with $r(2)=0$ and $r(3)=\B$ (likewise the $r$ above), we have $r(1)=\B$. Indeed $r''(1)$ is either $0$ or $\B$ by Lemma \ref{lem:differ not by c}. On the other hand, $r''(1)\neq 0$ because otherwise we would produce the string $0\B\B$ in $T_1$, which is also in $T_0$. Hence $r''(1)=\B$. Since in $T_0$ we have the row $r'$ with $r'(3)=\A$, there exists a row $r''$ in $T_1$ with $r''(3)=\A$. If $r''(2)=0$ we are done, as we produce $\A 0\A$ in both tables. For $r''(2)=\A$ we have the cubic move in $T_1$, $\A\B\C + \B 0\B+\A\A = \A 0\A +\B\A\C + \B\B$. For $r''(2)=\B$, we have the quartic move in $T_1$,
$000+\A\B\C+\B 0\B = \A 0\A + \B\B 0 +0\B\B+ 0\C$. For $r''(2)=\C$, we have the quartic move in $T_1$, $000+\A\B\C + \B 0\B+\C\A = \A 0\A+\B\B 0+0\C\C + 0\B$. 
\end{proof}

\subsubsection{The case of $n=6$ leaves}\label{part3}

After having set up the cornerstone of our approach, we are ready to first establish the case of $n=6$ leaves. Let $P$ be the lattice polytope of the Kimura $3$-parameter model for $n=6$ leaves. Here we are in the setting of polytopes. To be consistent with standard terminology, binomials in the ideal of the Kimura $3$-parameter model are identified with relations among lattice points, which in turn are naturally identified with variables. The minimal generating relations among the vertices of the polytope $P$ constitute a {\it Markov basis}. The degree of an element of a Markov basis is the total degree of the corresponding binomial in the standard grading. The degree of the corresponding table is the number of rows. Only in this section, given a Markov basis element $B$, which we think of as a binomial, we introduce the notation $\deg(B)$ to denote its degree. \\
\indent As recalled in Section \ref{notation}, the polytope $P$ is $18$ dimensional. Following the notation of Section \ref{notation}, a generating set of the full lattice $M^6$ is $e_{(i,g)}\in [6]\times G$. However, our lattice is a sublattice of $M^6$. Since we have the six linear relations $e_{(i,0)}^*+e_{(i,\A)}^*+e_{(i,\B)}^*+e_{(i,\C)}^*=1$ for $1\leq i\leq 6$ satisfied by the vertices of the polytope, we can choose the elements $e_{(i,\A)}, e_{(i,\B)}, e_{(i,\C)}$ for $1\leq i \leq 6$ to serve as a basis of the $18$-dimensional lattice of interest. 

\begin{proposition}\label{prop:n6}

The polytope $P$ defines an $18$ dimensional projectively normal (in particular, Cohen-Macaulay) toric variety in $\PP^{1023}$. Its Hilbert series is 
$Hs(t) = \frac{N(t)}{(1-t)^{19}}$, where 
$$
\begin{small}
\begin{matrix}
N(t) =  t^{15} + 1005t^{14}+230763t^{13}+11423223t^{12}\\
          +197336781t^{11} +1476133641t^{10} + 5369113631t^{9}\\
          +10097960379t^{8}+10077653595t^{7} +5323111487t^6 \\
          +1442513865t^5+187603341t^4+10384023t^3+198795t^{2}+1005t+1.\\
\end{matrix}
\end{small}
$$
\noindent Its Hilbert polynomial is
$$
\begin{small}
\begin{matrix}
H(t) = \frac{22261501}{4168212048000}t^{18} +\frac{799045380}{4168212048000}t^{17}+\frac{13381457673}{4168212048000}t^{16}+\frac{138721353336}{4168212048000}t^{15} \\
\end{matrix}
\end{small}
$$
$$
\begin{small}
\begin{matrix}
+\frac{995839168812}{4168212048000}t^{14}+\frac{5247736051320}{4168212048000}t^{13}+\frac{21011354421226}{4168212048000}t^{12}+\frac{65366574541632}{4168212048000}t^{11}\\
\end{matrix}
\end{small}
$$
$$
\begin{small}
\begin{matrix}
+\frac{160636901283573}{4168212048000}t^{10}+\frac{316408365264420}{4168212048000}t^9+\frac{507035368484229}{4168212048000}t^{8}+\frac{671227146881928}{4168212048000}t^{7}\\
\end{matrix}
\end{small}
$$
$$
\begin{small}
\begin{matrix}
+\frac{744003206327314}{4168212048000}t^{6}+\frac{695859081785280}{4168212048000}t^{5}+\frac{545170528162872}{4168212048000}t^{4}+\frac{340981469563104}{4168212048000}t^{3}\\
\end{matrix}
\end{small}
$$
$$
\begin{small}
\begin{matrix}
+\frac{151089754960800}{4168212048000}t^{2}+\frac{38894674089600}{4168212048000}t+1.\\
\end{matrix}
\end{small}
$$
\noindent In particular, the Markov basis has elements of degree at most $16$. \\
\indent Let us consider the following two codimension two faces of $P$:
\begin{enumerate}
\item [(i)] $\tilde P$ contains points corresponding to flows that have $0$ or $\A$ on the sixth leaf. This is the intersection of $P$ with the linear subspace $e_{(6,\B)}^*=e_{(6,\C)}^*=0$.
\item[(ii)] $\tilde P'$ contains points corresponding to flows that do not have $\C$ on the sixth leaf and on the fifth leaf. This is the intersection of $P$ with the linear subspace $e_{(5,\C)}^*=e_{(6,\C)}^*=0$.
\end{enumerate}

\noindent The Hilbert series of (i) is $Hs(t) = \frac{\tilde N(t)}{(1-t)^{17}}$, where

$$
\begin{small}
\begin{matrix}
\tilde N(t) = t^{13}+1007t^{12}+107752t^{11}++2813176t^{10}\\
          +26622909t^9+109147219t^8+211160560t^7+ 199302992t^6 \\
          +91202787t^5+19336749t^4+1724040t^3+54360t^2+495t+1.\\
\end{matrix}
\end{small}
$$
\\
The Hilbert series of (ii) is $Hs(t) = \frac{\tilde N'(t)}{(1-t)^{17}}$, where
$$
\begin{small}
\begin{matrix}
\tilde N'(t) = 3t^{13}+2253t^{12}+211288t^{11}++5060488t^{10}\\
          +44891401t^9+174437831t^8+321990512t^7+ 291183248t^6 \\
          +127959653t^5+26052683t^4+2223560t^3+66520t^2+559t+1.\\
\end{matrix}
\end{small}
$$
\noindent In particular, the Markov basis in both cases has elements of degree at most $14$.

\end{proposition}

\begin{proof}
The computation of Hilbert series and verification of normality were obtained using \texttt{Normaliz} \cite{Normaliz}. The statements about the degree of Markov basis are a consequence of well-known theorems on regularity of normal toric varieties, see Appendix \ref{sec:App}.
\end{proof}

\begin{lemma}\label{prop:n6faces}
The following three codimension three faces $P_1, P_2, P_3$ of $P$ have Markov basis with elements of degree at most four:
\begin{enumerate}
\item [(i)] $P_1$ contains points corresponding to flows that have $0$ on the sixth leaf. This is the intersection of $P$ with the linear subspace $e_{(6,\A)}^*=e_{(6,\B)}^*=e_{(6,\C)}^*=0$ and is isomorphic to the Kimura $3$-parameter model polytope for five leaves.
\item[(ii)] $P_2$ contains points corresponding to flows that do not have $\B$ or $\C$ on the sixth leaf and do not have $\C$ on the fifth leaf. This is the intersection of $P$ with the linear subspace $e_{(5,\C)}^*=e_{(6,\B)}^*=e_{(6,\C)}^*=0$.
\item[(iii)] $P_3$ contains points corresponding to flows that do not have $\C$ on the fourth, the fifth and the sixth leaf. This is the intersection of $P$ with the linear subspace $e_{(4,\C)}^*=e_{(5,\C)}^*=e_{(6,\C)}^*=0$.  
\end{enumerate} 
\end{lemma}
\begin{proof}
We employed \texttt{4ti2} \cite{4ti2} to compute explicitly the Markov basis in all three cases. More specifically, for $P_2$ we obtained $47112$ relations: $36840$ quadrics, $2304$ cubics, and $7968$ quartics. For $P_3$, we obtained $57058$ relations: $48600$ quadrics, $2176$ cubics, and $6282$ quartics. 
\end{proof}

\begin{remark}
The polytopes $P_1, P_2$ and $P_3$ are not isomorphic, although they have the same dimension. One can easily see that $P_1, P_2, P_3$ have  $256, 384, 432$ vertices respectively. Similarly, $\tilde P$ and $\tilde P'$ have $512$ and $576$ vertices respectively.
\end{remark}

Let us consider a Markov basis element $B$ of $P$. We show that one of the following holds: 
\begin{enumerate}
\item[(i)] $B$ has either degree less than or equal to four;
\item[(ii)] $B$ has $\deg(B) > 16$, which is not possible by Proposition \ref{prop:n6};
\item[(iii)] $B$ is a Markov basis element of $\tilde P$ or $\tilde P'$ of degree at least $15$, which is not possible by Proposition \ref{prop:n6};
\item[(iv)] $B$ is a Markov basis element for a polytope isomorphic to $P_1, P_2$ or $P_3$ (in this case, it has degree at most four by Lemma \ref{prop:n6faces}).
\end{enumerate}

\begin{proposition}
Any Markov basis element $B$ for $P$ has degree at most four.
\end{proposition}
\begin{proof}

It is enough to restrict to the main case (\ref{maincase}). We first prove two claims, Claim (i) and (ii). \\
\noindent Claim (i): For any row $r$ of $T_0$ distinct from the first one, 
for any pair of indices $2< i<j\leq 6$, we have that either $\phi_{\C}(r_{0\B}(i))=\phi_{\C}(r(i))$ or $\phi_{\C}(r_{0\B}(j))=\phi_{\C}(r(j))$. The analogous statement holds for $T_1$, with the group homomorphism $\phi_{\C}$ replaced by $\phi_{\B}$.  
\begin{proof}[Proof of Claim (i)]
Suppose the statement is not true for some pair of indices $i,j$. If $r_{0\B}(i)-r(i)=r_{0\B}(j)-r(j)$, then we can make a quadratic move on $i,j$, and conclude using the Difference Lemma \ref{lem:difference}. Thus, without loss of generality, we may assume $r_{0\B}(i)-r(i)=\A$ and $r_{0\B}(j)-r(j)=\B$. If there exists another index $2<k\leq 6$ such that $r_{0\B}(k)-r(k)\neq 0$, then we can make a move on a subset of $\{i,j,k\}$ and, again, conclude by the means of the Difference Lemma \ref{lem:difference}. In conclusion, $\sum_{l=3}^6r_{0\B}(l)=\A+\B+\sum_{l=3}^6r(l)$. As $r$ and $r_{0\B}$ are flows and $r_{0\B}(1)+r_{0\B}(2)=\B$, this contradicts Lemma \ref{lem:differ not by c}, which prescribes the first two columns of a row differing {\it not} by $\C$ with $r_{0\B}$. 
\end{proof}
\indent By Proposition \ref{aaaaa,aaaa,aaabc,aabb,aaab} and Lemma \ref{lem:agreeon0}, we may assume that $r_{0\B}(6)=r_{\A\B}(6)=0$, as the disagreement string between the two rows has length at most three, outside the first two columns. \\ 
\noindent Claim (ii): There exists at most one index $i>2$ such that $\phi_{\C}(r_{0\B}(i))\neq 0$.
\begin{proof}[Proof of Claim (ii)]
As the number of such indices must be odd it is enough to prove that not all $r_{0\B}(3), r_{0\B}(4), r_{0\B}(5)$ are equal to $\A$ or $\B$. Not all can be equal to $\B$ since, by by Lemma \ref{lem:b->0}, that would contradict the fact that $r_{\A\B}$ is a flow. Say $r_{0\B}(3)=\B$ and $r_{0\B}(4)=r_{0\B}(5)=\A$. Then we have $r_{\A\B}(3)=0$ by Lemma \ref{lem:b->0} and thus $r_{\A\B}(4)+r_{\A\B}(5)=\C$. However, we may exclude $\{r_{\A\B}(4),r_{\A\B}(5)\}=\{\A,\B\}$ by Lemma \ref{lem:agreeon0} and we may exclude $\{r_{\A\B}(4),r_{\A\B}(5)\}=\{0,\C\}$ by Lemma \ref{lem:b->0}.
\end{proof}

\indent To continue our proof, we need to introduce some terminology, which we will use only here. A column index $1\leq i\leq 6$ is of type: 
\begin{enumerate}
\item[(a)] if all elements of $G$ appear in the corresponding $i$th column of $T_0$ (and of $T_1$); 
\item[(b)] if exactly three elements of $G$ appear in the $i$th column; 
\item[(c)] if exactly two elements of $G$ appear in the $i$th column;
\item[(d)] if exactly one element of $G$ appears in the $i$th column. 
\end{enumerate}
\noindent {\bf Step 0}: We suppose that all columns are of type $(a)$. \\
\indent By Claim (ii), there exists one index $j>2$ such that $\phi_{\C}(r_{0\B}(j))\neq 0$. For $i>2$, $i\neq j$ there must exist at least two rows $r_{i,1}, r_{i,2}$ such that $r_{i,1}(i)=r_{0\B}(i)+\A$ and $r_{i,2}(i)=r_{0\B}(i)+\B$. Note that $r_{i,1}$, $r_{i,2}$ are not the first row. Further, for the index $j$ there must exist one row $r_{j,1}$ different from the first one such that $r_{j,1}(j)=r_{0\B}(j)+\A$ or $r_{j,1}(j)=r_{0\B}(j)+\B$. 
All these rows are distinct by Claim (i). Hence, we obtain seven rows; we call them \emph{difference rows} for $T_0$. Note that the \emph{difference rows} for $T_0$ may only have $\A$ and $\B$ in columns $1,2$ by Lemma \ref{lem:differ not by c}. Analogously, we obtain at least seven \emph{difference rows} in $T_1$, with copies of $0$ or $\B$ in columns $1$ and $2$. \\
\indent If there exist \emph{difference rows} in $T_0$ and $T_1$ with $\B\B$ in the first two columns, then we obtain the string $\B\B 0$ in both tables and we conclude by induction on the degree of $B$. \\ 
\indent Thus suppose that there is no string $\B\B$ in columns $1,2$ of $T_1$. It follows that there must be at least seven copies of $0$ in columns $1,2$ in the \emph{difference rows} of $T_1$. Consequently, there are at least nine copies of $0$ in columns $1,2$ in $T_1$. By Lemma \ref{lem:forbidden configurations}, there is no string $00$ in columns $1,2$ in $T_0$, and the \emph{difference rows} for $T_0$ do not have copies of $0$ in columns $1,2$. In conclusion, we have at least this amount of distinct rows in $T_0$: 
\begin{enumerate}
\item[(i)] three, that are the first ones; 
\item[(ii)] seven, that are the \emph{difference rows};
\item[(iii)] seven, that contain copies of  $0$ in column $1$ or $2$;
\item[(iv)] two, that have $\C$ in column $1$ or $2$.
\end{enumerate}
Then, we have $\deg(B)> 18$.  This is impossible for a Markov basis element by Proposition \ref{prop:n6}.\\

\noindent {\bf Step 1}: We suppose that there exists exactly one column of type $(b)$ and all others are of type $(a)$. We may proceed as before, however we obtain only six \emph{difference rows} in the case when the column of type $(b)$ has column index $3\leq i \leq 6$. In the case when the column index of the column of type $(b)$ is either $1$ or $2$, we obtain seven \emph{difference rows}, but we cannot assume that there exists an additional row with $\C$ in the same column index of the column of type $(b)$. In either of these cases, we have $\deg(B)\geq 3+2\times 6+2=17$, that contradicts Proposition \ref{prop:n6}.\\

\noindent {\bf Step 2}: We suppose that there exist exactly two columns of type $(b)$ (resp.~ one column of type $(c)$). Here, we obtain five \emph{difference rows}. However, $B$ represents a Markov element for $\tilde P'$ (resp.~$\tilde P$), whose ideals have regularity $14$; see Appendix \ref{sec:App} for the definition of the associated ideal. We obtain the bound $\deg B\geq 3+2\times 5+2=15>14$ which contradicts Proposition \ref{prop:n6}. \\

\noindent {\bf Step 3}: We suppose there exist either: 

\begin{enumerate}

\item[(i)] three columns of type $(b)$), or 
\item[(ii)] one column of type $(b)$ or $(c)$ and one column of type $(c)$, or 
\item[(iii)] one column of type $(d)$. 

\end{enumerate}
\noindent In such cases we conclude by Lemma \ref{prop:n6faces}.
\end{proof}

\subsubsection{Proof of the main case}

In this last part, we finish our proof dealing with the main case (\ref{maincase}). This will be done uniformly, i.e., with the same arguments in all the three instances of the main case and only technical details differ. Here the number of leaves is $n\geq 7$. The outline is as follows:

\begin{enumerate}
\item[(i)] We show that, if $r_{0\B}(i)=r_{\A \B}(i)$, then we have $r_{0\B}(i)=r_{\A \B}(i)=0$;

\item[(ii)] Among the pairs of tables we consider (tables where we have fixed the first two entries of the rows $r_{0\B}$ and $r_{\A \B}$ and performed moves of degree at most four so that $r_{0\B}$ and $r_{\A \B}$ have the agreement string as large as possible) using at most moves of degree four, we attain the situation where the number of {\it bad pairs}, i.e., strings $xy$, with $x,y\neq 0$, in columns $n-1$ and $n$ is as small as possible;

\item[(iii)] We show that we can kill all the {\it bad pairs}, i.e., we can make moves of degree at most four killing all of them. Summing up the two columns indexed by $n-1$ and $n$ allows us to conclude by induction on the number of leaves $n$; see Theorem \ref{n->n-1 by badpairs}. 

\end{enumerate}

We are now ready to establish the main case in the following lemmas. 

\begin{lemma}
We may assume that no rows in $T_0$ has the string $\A\A$ or $\B\B$ in columns $n-1$ and $n$. Analogously, no row in $T_1$ has the string $\A\A$ or $\C\C$ in columns $n-1$ and $n$.
\end{lemma}
\begin{proof}
In such a case we make a quadratic move in columns $n-1$ and $n$ and we conclude by applying the Difference Lemma \ref{lem:difference}.
\end{proof}
\begin{lemma}
We may assume that no row in $T_0$ has the string $\A\B$ or $\B\A$  in columns $n-1$ and $n$. Analogously, no row in $T_1$ has the string $\A\C$ or $\C\A$  in columns $n-1$ and $n$.
\end{lemma}
\begin{proof}

Let $r$ be such a row with such a string in columns $n-1$ and $n$ of $T_0$. If for some other column index $i>2$, we have $r(i)\neq r_{0\B}(i)$ then we may exchange $i$ and a nonempty subset of elements under the agreement string. Then we conclude by applying the Difference Lemma \ref{lem:difference}. As $r$ is a flow, we have $r(1)+r(2)=\A$. This contradicts Lemma \ref{lem:differ not by c}.
\end{proof}
\begin{lemma}
We may assume that under the agreement string no row in $T_0$ has $\C\C$. Analogously, no row in $T_1$ has $\B\B$.
\end{lemma}
\begin{proof}

Let $r$ be a row in $T_0$ with $\C\C$ under the agreement string. We first claim we may assume that $r_{\A\B}$ does not have $\C$ in any column. For the sake of contradiction, suppose $r_{\A\B}(i)=\C$ for some column index $i$. Whence, by Lemma \ref{lem:b->0}, we have $r_{0\B}(i)=0$. By compatibility of the tables $T_0$ and $T_1$, there exists a row $r'$ in $T_0$ with $r'(i)=\C$. By the Standard Lemma \ref{standard argument}, we can make a move to obtain $r_{0\B}(i)=\C$ and conclude by applying Lemma \ref{lem:agreeon0}.  \\
\indent We divide the rest of the proof into two steps according to whether or not there exists $\B$ in $r_{0\B}$. \\

\noindent {\bf Step 1}: Suppose there exists another $\B$ in $r_{0\B}$. The tables $T_0$ and $T_1$ are the following: 

$$
T_0-T_1=\begin{bmatrix}
\A & \A  & 0 & \ldots & 0 & 0 \\
0 & \B & \B & \ldots & 0 & 0\\
\ldots & \ldots & \ldots & \ldots & \C &\C\\
\ldots & \ldots & \ldots & \ldots & \ldots &\ldots \\
\end{bmatrix}-\begin{bmatrix}
0 & 0 & 0 & \ldots & 0 & 0 \\
\A & \B & 0 & \ldots & 0 & 0\\
\ldots & \ldots & \ldots & \ldots &\ldots & \ldots \\
\ldots & \ldots & \ldots & \ldots &\ldots & \ldots \\
\end{bmatrix}.
$$
\noindent By Lemma \ref{lem:differ not by c}, the counting function $\A_3+\B_3-0_{12}-\C_{12}$ is nonnegative on $T_0$. Let $r$ be a row in $T_1$, where the function is strictly positive. We now exclude the case $r(3)=\B$. Indeed, in this case, if $r(2)=\B$, then we obtain $0\B\B$ in both tables. If $r(2)=\A$, by the positivity of the counting function on $r$, we have $r(1)=\B$ and we may perform a quadratic move to obtain the string $0\B\B$. Whence $r(3)=\A$. \\
\indent Let $r'$ be a row in $T_0$ with $r'(3)=\A$. By the Standard Lemma \ref{standard argument}, we can make a move between $r'$ and $r_{0\B}$ involving this entry. In particular, if $r(1)+r(2)=\C$, then we make a quadratic move between $r$ and $r_{\A\B}$ on first two entries and conclude by Lemma \ref{lem:agreeon0}. Thus $r(1)=r(2)=\B$. Let $r''$ be a row in $T_1$ with $r''(3)=\B$. We finish the proof of {\bf Step 1} by proving that we can always obtain $0\B\B$ in $T_1$.
First, suppose $r''(j)=\A$ for $j=1$ or $2$. Then we may exchange $r''$ with $r$ on column indices $j$ and $3$, obtaining a row $\tilde r$ such that $\tilde r(1)+\tilde r(2)=\C$ and $\tilde r(3)=\B$. Then we can make a quadratic move between $\tilde r$ and $r_{\A\B}$ to obtain $0\B\B$ in both tables. Also, notice that $r''(2)\neq \B$ as this immediately leads to $0\B\B$ in both tables. If $r''(1)=\C$ we may exchange $r''$ and $r_{\A\B}$ on column indices $1$ and $3$, obtaining $0\B\B$ in both tables. If $r''(1)=r''(2)=0$, we can make a quadratic move between $r''$ and $r$. Similarly, if $r''(1)=0$ and $r''(2)=\C$ we can make an exchange with $r_{\A\B}$. 
If $r''(1)=\B$ and $r''(2)=0$ we first exchange it with $r_{\A\B}$ on column indices $2,3$, then we apply $\A 0\B+\B\B\A=\B 0\A+\A\B\B$. Finally, if $r''(1)=\B$ and $r''(2)=\C$, we apply the cubic move $$000+\B\B\A+\B\C\B=0\B\B+\B\C 0+\B 0\A.$$
%
%
%
%

\noindent {\bf Step 2}: Suppose there is no $\B$ in $r_{0\B}$; without loss of generality we may assume we have $\A$ and $\C$ in columns $3,4$. In column $3$, in row $r_{\A\B}$ of $T_1$ we cannot have $\A$ by Lemma \ref{lem:agreeon0}; moreover, we cannot have $\B$ by the Standard Lemma \ref{standard argument} applied to table $T_0$, as we would produce $\B$ in the row $r_{0\B}$, contradicting Lemma \ref{lem:agreeon0}. Thus we have $0$ in column $3$ in the row $r_{\A\B}$, since $\C$ is excluded in row $r_{\A\B}$ by the claim in the very first part of the proof. Since the disagreement string has length at most three by Corollary \ref{Reduction of four}, we have the following tables $T_0$ and $T_1$: 
$$
T_0-T_1=\begin{bmatrix}
\A & \A  & 0 & 0 & 0 &  \ldots & 0 & 0 \\
0 & \B & \A & \C & x &  \ldots & 0 &0\\
\ldots & \ldots & \ldots &\ldots &\ldots &\ldots &  \C &\C  \\
\ldots & \ldots & \ldots &\ldots &\ldots  & \ldots & \ldots &\ldots\\
\end{bmatrix}-\begin{bmatrix}
0 & 0 & 0 & 0 & 0 &  \ldots & 0 & 0 \\
\A & \B & 0 & y & z   & \ldots& 0&0 \\
\ldots & \ldots & \ldots & \ldots &\ldots& \ldots &\ldots & \ldots \\
\ldots & \ldots & \ldots & \ldots &\ldots& \ldots &\ldots & \ldots \\
\end{bmatrix}.
$$
\noindent Furthermore $\{y,z\}=\{\A,\B\}$. By the disagreement string length, we have $x=0$. Consider the group morphism $\phi_{\C}:G\rightarrow \ZZ_2$ and apply it to columns $3,4,5$. Note that the evaluation of $r_{0\B}$ under $\phi_{\C}$ in column indices $3,4,5$ is the $0/1$ vector $(1,0,0)$. We claim that no row of $T_0$ can differ by more than one element with respect to $r_{0\B}$ in column indices $3,4,5$. Indeed, suppose a row $r$ in $T_1$ differs on $i,j\in \{3,4,5\}$. Then $r(i)+r(j)-r_{0\B}(i)-r_{0\B}(j)\in\{0,\C\}$. Thus, by the Standard Lemma \ref{standard argument}, we can make a quadratic move on $i,j$ and conclude by Difference Lemma \ref{lem:difference}. By double counting, there must exist a row $r'$ in $T_1$ such that $r'(3)\in \{\A,\B\}$ and $r'(4), r'(5)\in \{0,\C\}$. By a quadratic move and the claim at the very first part of the proof, we may assume $r'(4)=r'(5)$. Now we can make a quadratic move between $r'$ and $r_{\A\B}$ involving the entry in column $3$ and the entry in either column $4$ or $5$. However, we may conclude as in the first part of {\bf Step 2}. 
\end{proof}

\begin{lemma}\label{lem:noacinT0}
We may assume that under the agreement string no row in $T_0$ has $\A\C$ or $\C\A$ (resp.~$\B\C$ or $\C\B$).
\end{lemma}
\begin{proof}
\noindent {\bf Step 0}: Assume that there exists $\B$ in $r_{0\B}$ and $\C$ in $r_{\A\B}$; without loss of generality we may assume 
that they are in columns $3,4$. In this case the tables are:

$$
T_0-T_1=\begin{bmatrix}
\A & \A  & 0 & 0 &\ldots & \ldots & 0 & 0 \\
0 & \B & \B & 0 &\ldots &\ldots & 0 & 0  \\
\ldots & \ldots & \ldots&\ldots &\ldots &  \ldots & \A (\textnormal{resp.}~\B) & \C   \\
\ldots & \ldots & \ldots & \ldots&\ldots&\ldots&\ldots &\ldots \\
\end{bmatrix}-\begin{bmatrix}
0 & 0 & 0 & 0 & \ldots &  \ldots & 0 & 0 \\
\A & \B & 0 & \C & \ldots & \ldots & 0  & 0\\
\ldots & \ldots & \ldots & \ldots&\ldots&\ldots&\ldots &\ldots \\
\ldots & \ldots & \ldots & \ldots&\ldots&\ldots&\ldots &\ldots \\
\end{bmatrix}. 
$$

\noindent Let $r$ be the row in $T_0$ that contains the string $\A\C$ (resp.~$\B\C$) under the agreement string. By Lemma \ref{lem:agreeon0} and Lemma \ref{lem:differ not by c}, we see that $r(4)=0$. Let $r'$ be a row in $T_0$ such that $r'(4)=\C$. By Lemma \ref{lem:agreeon0}, we can exclude $\C$ under the agreement string. Furthermore, performing a quadratic move, we notice that if $r'$ has $00$ under the agreements string, we could reduce $\A\C$ (resp.~$\B\C$) to $\A0$ (resp.~$\B0$), contradicting the minimality of the number of  {\it bad pairs}. Also, $r'$ cannot have $0\B$ or $\B 0$ (resp.~$0\C$ or $\C 0$) under the agreement string, as we could exchange it with $\A\C$ (resp.~$\B\C$) and conclude as before. Thus, under the agreement string, $r'$ has either the string $0\A$ or $\A 0$ (resp.~$0\B$ or $\B 0$). Now, Lemma \ref{lem:agreeon0} and Lemma \ref{lem:differ not by c} allow us to conclude that $r'(3)=\B$. Hence, the counting function $\B_3-\C_4$ is strictly positive in $T_0$. Let $r''$ be a row in $T_1$ such that $r''(3)=\B$ and $r''(4)\neq \C$. By Lemma \ref{lem:agreeon0}, we may exclude $\A$ in column $4$ in $r''$. Consequently, by Lemma \ref{lem:differ not by c}, $r''$ has either $0$ or $\B$ in column $2$. If $r''(2)=\B$, we obtain the same string $0\B\B0$ in both tables. If $r''(2)=0$, we obtain the  string $\A0\B\C 0$ in $T_1$; we now show we may also obtain it in $T_0$. We discuss this according to the three crucial cases:   
\begin{enumerate}
\item[(i)] {\bf Case I} and {\bf II}: We apply the move 
$\A\A 00+0\B\B 000+\A 0+??\B\C xy=\A 0\B\C+\A\B\B 0xy+0\A+??0000$, where $xy$ is under the agreement string and $x+y=\A$. (resp.~We consider the first two entries of $r'$, which by Lemma \ref{lem:differ not by c} could be: $\A\A$, $\A\B$, $\B\A$, $\B\B$. The last three allow to obtain $\A\B 0\C$ in both tables. As $r'$ must agree on all nonspecified entries with $r_{0\B}$ this contradicts the fact that $r'$ is a flow.); 

\item[(ii)] {\bf Case III}: We apply the move $\A\A 00+0\B\B 000+\B 0+??\B\C xy=\A 0\B\C+\B\A\B 0xy+0\B+??0000$ where $x+y=\A$. (resp.~We proceed as before, noting that we do not use the third row, except for $\B\A$, in which case we obtain $\B\A 0\C$ in both tables). \\

\end{enumerate}

\noindent {\bf Step 1}: Assume there exists $\B$ in  $r_{0\B}$ and no $\C$ in $r_{\A\B}$. The tables are:
$$
T_0-T_1=\begin{bmatrix}
\A & \A  & 0 & 0 & 0 & \ldots & 0 & 0 \\
0 & \B & \B & x & y & \ldots & 0 & 0 \\
\ldots & \ldots & \ldots & \ldots &\ldots &\ldots &\A (\textnormal{resp.}~\B) & \C \\
\ldots & \ldots & \ldots & \ldots& \ldots & \ldots & \ldots &\ldots\\
\end{bmatrix}-\begin{bmatrix}
0 & 0 & 0 & 0 & 0 & \ldots & 0 & 0 \\
\A & \B & 0 & \A & \B & \ldots & 0 & 0 \\
\ldots & \ldots & \ldots & \ldots& \ldots & \ldots & \ldots &\ldots\\
\ldots & \ldots & \ldots & \ldots& \ldots & \ldots & \ldots &\ldots\\
\end{bmatrix}.
$$
As the disagreement string is of length at most three, we must have $x=y$. Further, by Lemma \ref{lem:agreeon0} $x=y=0$ or $x=y=\C$. 
Consider the group morphism $\phi_{\C}:G\rightarrow \ZZ_2$. We claim that after applying $\phi_\C$ to column indices $3,4,5$, no row can differ on more than one index from $\phi_\C((\B,x,x))=(1,0,0)$. Indeed, if a row $\tilde r$ differs on two indices $i,j$, then, by the Difference Lemma \ref{lem:difference}, we may assume $r_{0\B}(i)=\tilde r(i)+\A$ and $r_{0\B}(j)=\tilde r(j)+\B$. The rows $r$ and $\tilde r$ must differ by $\A$ either in column index $i$ or $j$, and by $\B$ on the other. In particular, by reducing the number of {\it bad pairs} $\A\C$ (resp.~$\B\C$) under the agreement string, we exclude the situation when $\tilde r$ has $00$ under the agreement string. By the Difference Lemma \ref{lem:difference}, we also know that $\C$ does not appear in $\tilde r$ under the agreement string. In the same way, if $\A$ or $\B$ appears under the agreement string, we may exchange it along with the index $i$ or $j$, again contradicting Difference Lemma \ref{lem:difference}. By double counting, there exists a row $\tilde r'$ in $T_1$ such that $\phi_\C((\tilde r'(3, 4,5)))=(1,0,0)$. In particular, there exist two indices such that we can make a quadratic move between $\tilde r'$ and $r_{\A\B}$. This either contradicts Lemma \ref{lem:b->0} or one decreases the Hamming distance.\\

\noindent {\bf Step 2}: Assume there is no $\B$ in  $r_{0\B}$ and there exists $\C$ in $r_{\A\B}$. The tables are:
$$
T_0-T_1=\begin{bmatrix}
\A & \A  & 0 & 0 & 0 & \ldots & 0 & 0 \\
0 & \B & \A & \C & 0 & \ldots & 0 & 0 \\
\ldots & \ldots & \ldots & \ldots &\ldots &\ldots & \A (\textnormal{resp.}~\B) & \C \\
\ldots & \ldots & \ldots & \ldots & \ldots & \ldots  & \ldots &\ldots \\
\end{bmatrix}-\begin{bmatrix}
0 & 0 & 0 & 0 & 0 & \ldots & 0 & 0 \\
\A & \B & x & y & \C & \ldots & 0 & 0 \\
\ldots & \ldots & \ldots & \ldots & \ldots & \ldots  & \ldots &\ldots \\
\ldots & \ldots & \ldots & \ldots & \ldots & \ldots  & \ldots &\ldots \\
\end{bmatrix}.
$$
As before $x=y$ equals $\B$ or $0$. Further $r(5)=0$. We apply $\phi_{\B}$ to column indices $3,4,5$.\\
\indent We claim we may assume that no row $\tilde r$ in $T_0$ differs from $\phi_{\B}((\A,\C,0))=(1,1,0)$ on more than one index.
For the sake of the contradiction, suppose there exists $\tilde r$ in $T_0$ differing on $i$ and $j$. If $r_{0\B}(i)-\tilde r(i)=r_{0\B}(j)-\tilde r(j)$, then we make a quadratic move between $r_{0\B}$ and $\tilde r$ on $i,j$. If the difference equals $\A$, we conclude by the Difference Lemma \ref{lem:difference}. Thus we assume the difference equals $\C$. If $5\in\{i,j\}$ we conclude by Lemma \ref{lem:agreeon0}. Hence, $\{i,j\}=\{3,4\}$; on the other hand, this reduces the Hamming distance. Consequently we have $r_{0\B}(i)-\tilde r(i)=\A$ and $r_{0\B}(j)-\tilde r(j)=\C$. Notice that we cannot have $r(i)-\tilde r(i)=r(j)-\tilde r(j)=\A$, thus at least one difference must be equal to $\C$. Hence, we exclude $00$ in $\tilde r$ under the agreement string, as then we could reduce the number of $\A\C$ (resp.~$\B\C$) under the agreement string. Further, $\A$ and $\B$ also cannot appear under the agreement string, as otherwise we may conclude by the Difference Lemma \ref{lem:difference}. Whence $\tilde r$ has $0\C$ or $\C 0$ under the agreement string. By Lemma \ref{lem:agreeon0}, we have $j\neq 5$. Let $\tilde r'$ be a row of $T_0$  with $\tilde r'(5)=\C$. As before, we conclude that $\tilde r'$ has $\A 0$ or $0\A$ under the agreement string (resp.~$0\B$ or $\B 0$), and $\tilde r'(3)=\A$, $\tilde r'(4)=\C$. We now exclude the case $i=5$, i.e.,~$\tilde r(5)=\A$. In such a case, we could exchange $r$ and $\tilde r$ on column $5$ and under the agreement string; then with $\tilde r'$ on column indices $5$ and $j$; finally with $r_{0\B}$ on column indices $5$ and the last entry to conclude by Lemma \ref{lem:agreeon0}. (Resp.~We apply the relation on $5$ and the agreement string $000+\A 0\C+\C 0\B=\C 0 \C+0 0 \B+\A 00$.) \\
\indent In conclusion, our discussion leads to $\{i,j\}=\{3,4\}$ and $\tilde r(5)=0$. However, we may exchange $\tilde r$ with $\tilde r'$ on $5$ and $j$. Consequently we exchange with $r_{0\B}$ on $5$ and under the agreement string to conclude by Lemma \ref{lem:agreeon0}. This concludes the verification of our claim. \\
\indent By the claim, there must exist a row $r''$ in $T_1$, such that $\phi_{\B}(r''((3,4,5)))=(1,1,0)$. On two of these indices, $r''$ differs from $r_{\A\B}$ by the same element: either $\A$ or $\C$. We can make a quadratic move on these two column indices and conclude by Difference Lemma \ref{lem:difference}.\\

\noindent {\bf Step 3}: Assume there is no $\B$ in  $r_{0\B}$ and no $\C$ in $r_{\A\B}$. The tables are:

$$
T_0-T_1=\begin{bmatrix}
\A & \A  & 0 & 0 & 0 & \ldots & 0 & 0 \\
0 & \B & \A & \C & 0 & \ldots & 0 & 0 \\
\ldots & \ldots & \ldots & \ldots &\ldots &\ldots & \A (\textnormal{resp.}~\B) & \C \\
\ldots & \ldots & \ldots & \ldots& \ldots & \ldots& \ldots &\ldots \\
\end{bmatrix}-\begin{bmatrix}
0 & 0 & 0 & 0 & 0 & \ldots & 0 & 0 \\
\A & \B & x & y & z & \ldots & 0 & 0 \\
\ldots & \ldots & \ldots & \ldots& \ldots & \ldots& \ldots &\ldots \\
\ldots & \ldots & \ldots & \ldots& \ldots & \ldots& \ldots &\ldots \\
\end{bmatrix}.
$$
\noindent Suppose $x=\B$. Let $\tilde r$ be a row of $T_0$ with $\tilde r(3)=\B$. As in the previous steps, we may assume that $\tilde r$ has $0\A$ or $\A 0$ (resp.~$0\B$ or $\B 0$) under the agreement string and $\tilde r(i)=r_{0\B}(i)$ for $4\leq i \leq n-3$. By Lemma \ref{lem:differ not by c}, we have $\tilde r(1,2)=\A\A$ or $\B\B$ (resp.~ $\tilde r(1,2)=\A\B$ or $\B\A$; we may obtain $0\B\B$ in both tables by the move: $\A\A0+0\B\A\C00+(\A\B/\B\A)\B\C 0\B=0\B\B+(\A\B/\B\A)0\C 00+\A\A\A\C0\B$). However, $\B\B$ easily leads to $0\B\B$ in both tables by the cubic move $\A\A 0+0\B\A+\B\B\B = 0\B\B+\B\A 0+\A\B\A$ in $T_0$. Furthermore, we may assume that $\B\B$ does not appear on column indices $1,2$ in any row in $T_0$, otherwise we would exchange with $\tilde r$ obtaining $\B\B\B$ in columns $1,2,3$. It follows that $\A_{12}-0_3-\B_3$ is positive on $T_0$. However, a positive row in $T_1$ contradicts Lemma \ref{lem:differ not by c}.\\
\indent Thus we may assume $x=0$. Without loss of generality $\{y,z\}=\{\A,\B\}$. We apply the homomorphism $\phi_\C$ to column indices $3,4,5$. We prove that no row may differ on two indices from $\phi_{\C}(r_{0\B}(3,4,5))=(1,0,0)$ in $T_0$. This is analogous to {\bf Step 1}. Whence there exists a row $\tilde r$ in $T_1$, such that $\phi_{\C}(\tilde r((3,4,5)))=(1,0,0)$. We may assume $\tilde r(4)=\tilde r(5)$, as otherwise we can make a quadratic move on column indices $4,5$ and conclude by previous steps. However, in such a case we may exchange $\tilde r$ with $r_{0\B}$ (on column index $3$ and on column index either $4$ or $5$), conclude by Lemma \ref{lem:agreeon0} or reduce to the first part of this step, where we assume $x=\B$.
\end{proof}

\begin{lemma}
We may assume that under the agreement string no row in $T_1$ has $\A\B$ or $\B\A$ (resp.~$\B\C$ or $\C\B$).
\end{lemma}
\begin{proof}

Let us act on tables $T_0$, $T_1$ by the flow $(\A,\A, 0,\dots,0)\in \mathfrak G$ and then apply the group automorphism $\B\leftrightarrow \C$. This translates {\bf Case I} and {\bf Case III} to {\bf Case III} and {\bf Case I} of Lemma \ref{lem:noacinT0} respectively; cf.~Remark \ref{rem:symmetry}. However, {\bf Case II} is not transformed to the previous cases, due to the rows $r_{\A 0}$ in $T_0$ and $r_{\B\A}$ in $T_1$. We note that in {\bf Steps 1}, {\bf 2}, and {\bf 3} of Lemma \ref{lem:noacinT0} we are only using the rows $r_{0\B}$ in $T_0$ and $r_{\A\B}$ in $T_1$ that still appear after translating {\bf Case II}. \\
\indent Thus, we only need to conclude in {\bf Case II} and {\bf Step 0}, i.e.,~there exists $\B$ in $r_{0\B}$ and $\C$ in $r_{\A\B}$. Without loss of generality, we may assume that they are in columns $3,4$. The tables are:

$$
T_0-T_1=\begin{bmatrix}
\A & \A  & 0 & 0 &\ldots & 0 & 0 \\
0 & \B & \B & 0 &\ldots & 0 & 0 \\
\A & 0 & \ldots &\ldots &\ldots &\ldots &\ldots \\
\ldots & \ldots & \ldots&\ldots &\ldots &\ldots&\ldots \\
\end{bmatrix}-\begin{bmatrix}
0 & 0 & 0 & 0 & \ldots&0 & 0  \\
\A & \B & 0 & \C &\ldots & 0  & 0 \\
\B &\A &\ldots &\ldots &\ldots &\ldots &\ldots \\
\ldots & \ldots & \ldots & \ldots&\ldots&\A (\textnormal{resp.}~\C) & \B \\
\end{bmatrix}. 
$$

\noindent Let $r$ be the row in $T_1$ with a bad pair of the form $\A\C$ (resp.~$\B\C$). First we exclude $r(3)=\A,\B,\C$ by quadratic exchange with $r_{\A\B}$, and the Difference Lemma \ref{lem:difference} and Lemma \ref{lem:agreeon0}. Let $\tilde r$ be the row in $T_1$ such that $\tilde r(3)=\B$. 
We note that if $\tilde r(n-1)=\tilde r(n)=0$ then, exchanging with $r$ we could reduce the number of bad pairs. Moreover, by Lemma \ref{lem:agreeon0} we know that $\tilde r(n-1),\tilde r(n)\neq \B$. 
Furthermore, as we already know that $r(3)$ must be equal to zero, we have $\tilde r(n-1)+\tilde r(n)\neq r(n-1)+r(n)$. Thus, we must have $\{\tilde r(n-1),\tilde r(n)\}=\{0,\A\}$ (resp.~$\{\tilde r(n-1),\tilde r(n)\}=\{0,\C\}$).
Note that $\tilde r(2)=\B$ gives $0\B\B$ in both tables, thus we may assume $\tilde r(2)=0$, by Lemma \ref{lem:differ not by c}. 
Moreover, we have $\tilde r(4)=\C$ and hence the counting function $\B_3-\C_4$ is negative on $T_1$. Let $r'$ be the row in $T_0$ on which the function is negative, i.e.,~$r'(4)=\C$ and $r'(3)\neq \B$. Now, $r'(3)\neq \A$, as otherwise we exchange $r'$ and $r_{0\B}$ and conclude by Lemma \ref{lem:agreeon0}. Thus, by Lemma \ref{lem:differ not by c}, we have $r'(1),r'(2)\in \{\A,\B\}$. If $r'(2)=\B$ then we obtain $\A\B0\C$ in both tables, thus we may assume $r'(2)=\A$. We may obtain the flow $0\A\B\C$ in $T_0$, by exchanging $r'$ and $r_{0\B}$, and $\A 0\B\C$ in $T_0$, by exchanging with $r_{\A 0}$. We finish the proof by showing that we may obtain the latter 
in $T_1$, by the 
quadratic move $\A\B 0\C +? 0\B=\A 0\B\C+?\B 0$.
\end{proof}

\section{Appendix}\label{sec:App}

We present known algebraic results for algebras over monoids that are cones over normal lattice polytopes. Much more information can be found in \cite{BrGub, Eis, MS04, Stks, Vas04}.\\
\indent Let $M$ be a lattice and $P\subset \{1\}\times M\subset \ZZ\times M$ be a \emph{normal} lattice polytope generating the ambient lattice. Let $C(P)\subset \ZZ\times M$ be the cone over $P$. The cone $C(P)$, equipped with addition, has a natural structure of a graded monoid, with the grading induced by the first coordinate. The algebraic properties of the graded algebra $\CC[C(P)]$ are strongly related to combinatorial properties of $P$.
\begin{proposition}
The function $H_P:\ZZ_{\geq 0}\rightarrow\ZZ$ defined by $H_P(n)=|nP\cap \{n\}\times M|$ is a polynomial known as \textnormal{Ehrhart polynomial}. For all $n\geq 0$, it coincides with the Hilbert function (and hence with the Hilbert polynomial) of the algebra $\CC[C(P)]$. Moreover, it satisfies the \emph{Ehrhart reciprocity}, i.e.~$|H_P(-n)|=|\textnormal{int}(nP)\cap \{n\}\times M|$ for $n>0$, where $\textnormal{int}$ denotes the interior points of the polytope.
\end{proposition}
We immediately see that the polynomial $H_P(n)$ may agree with the Hilbert function even for negative $n$. This happens if and only if $H_P(n)=0$, as the algebra is positively graded. 
\begin{definition}[{\bf $a$-invariant, Hilbert regularity}]
The $a$-invariant $a(A)$ of an algebra $A$ is the largest integer $a$ such that the Hilbert function differs from the Hilbert polynomial. Hilbert regularity equals the $a$-invariant plus one.
\end{definition}
\begin{corollary}
The $a$-invariant of $\CC[C(P)]$ is always negative. It equals $-n$ for the smallest $n\in \ZZ_{>0}$ such that $nP$ contains an interior point.
\end{corollary}
\begin{proposition}
If $\dim P=d$, then $\sum_{j=0}^\infty H_P(j)t^j=h(t)/(1-t)^{d+1}$ for some polynomial $h$. The $a$-invariant of $\CC[C(P)]$ equals $\deg h -d-1$. 
\end{proposition}
We note that $d+1-\deg h$ is the smallest dilation of $P$ that contains an interior lattice point. 
\begin{proposition}[{\bf Hochster's Theorem}]
The algebra $\CC[C(P)]$ is Cohen-Macaulay.
\end{proposition} 
Throughout the article we were interested in generators of the ideal $I$ such that $\CC[C(P)]=\CC[x_p:p\in P\cap M]/I=S/I$. These are usually very hard to understand even for specific instances. However, there is an algebraic invariant that bounds their degree, known as {\it Castelnuovo-Mumford regularity}, or simply, the {\it regularity}. 

\begin{definition}[{\bf Castelnuovo-Mumford regularity}]
For an $S$-module $M$ its regularity $\textnormal{reg}(M)$ is defined as $$\textnormal{reg}(M)=\max\{j-i:b_{ij}\neq 0\},$$
where
$$0\leftarrow M\leftarrow \bigoplus_j S(-j)^{b_{0j}}\leftarrow\ldots \leftarrow \bigoplus_j S(-j)^{b_{ij}}\leftarrow \ldots \leftarrow 0$$
is the minimal free resolution of $M$. 
\end{definition}
As $I$ is an $S$ module, its regularity in particular bounds the degree of generators; this is the case $i=0$ in the definition. It can be seen that $\textnormal{reg}(\CC[C(P)])$ is the maximal degree of standard monomials under rev-lex in generic coordinates. Hence $\textnormal{reg}(I)$ bounds the degree of such a Gr\"obner basis, as $\textnormal{reg}(S/I)+1=\textnormal{reg}(I)$. The following proposition relates both notions of regularity introduced above.
\begin{proposition}
$a(M)\leq \textnormal{reg}(M)-\textnormal{depth}(M)$ and equality holds if $M$ is Cohen-Macaulay. In particular, $\textnormal{reg}(\CC[C(P)])=\deg h$ and $I$ is generated in degree at most $1+\deg h$.
\end{proposition}


\begin{small}

\noindent
{\bf Acknowledgements.}\smallskip \\
Mateusz Micha{\l}ek was supported by Polish National Science Centre grant no. 2015/19/D/ST1/01180, the Foundation for Polish Science (FNP) and is a member of AGATES group. The authors acknowledge the kind hospitality of UC Berkeley and FU Berlin, where this research was in part conducted.

\end{small}

\begin{small}

\end{small}

\bigskip

\noindent
\footnotesize {\bf Authors' addresses:}

\smallskip

\noindent Mateusz Micha{\l}ek,
Max Planck Institute for Mathematics in the Sciences, Leipzig, Germany\\
{\tt mateusz.michalek@mis.mpg.de}\\
Institute of Mathematics of
Polish Academy of Sciences, Warsaw, Poland\\
{\tt mmichalek@impan.pl}

\smallskip

\noindent Emanuele Ventura,
Max Planck Institute for Mathematics in the Sciences, Leipzig, Germany\\
{\tt emanuele.ventura@mis.mpg.de}

\end{document}